\documentclass[a4paper,11pt]{amsart}

\DeclareRobustCommand{\SkipTocEntry}[5]{}

\usepackage{amsmath}
\usepackage{bm}

\usepackage{upgreek}
\usepackage{ esint }
\usepackage{color}
\usepackage{comment}
\usepackage{amssymb}
\usepackage{amsfonts}

\usepackage{amscd}

\usepackage{slashed}
\usepackage{pdflscape}
\usepackage{tikz}
\usepackage{enumerate}
\usepackage{multirow}
\usepackage{stmaryrd}

\usepackage{bbold}
\usepackage[linkcolor=blue]{hyperref}
\usepackage{mathrsfs}

\newtheorem{theorem}{Theorem}[section]
\newtheorem{lemma}[theorem]{Lemma}
  
\newtheorem{proposition}[theorem]{Proposition}
\newtheorem{corollary}[theorem]{Corollary}

\theoremstyle{definition}
\newtheorem{definition}[theorem]{Definition}
\newtheorem{example}[theorem]{Example}

\theoremstyle{remark}
\newtheorem{remark}[theorem]{Remark}

\usepackage{pdfsync}

\usepackage{graphicx} 
\usepackage{amsmath} 
\usepackage{amsfonts}
\usepackage{amssymb}

\newcommand{\be}{\begin{equation}}

\newcommand{\ee}{\end{equation}}

\newcommand{\Odane}{\O}

\newcommand{\II}{{ I\hspace{-.8mm}I}}
\newcommand{\IIo}{\mathring{\!{ I\hspace{-.8mm} I}}{\hspace{.2mm}}}

\newcommand{\si}{\sigma}

\newcommand{\ba}{\begin{array}}

\newcommand{\ea}{\end{array}}

\newcommand{\beq}{\begin{eqnarray}}

\newcommand{\eeq}{\end{eqnarray}}

\newtheorem{lm}{lemma}

\newtheorem{thee}{theorem}

\newtheorem{proo}{proposition}

\newtheorem{co}{corollary}

\newtheorem{rem}{remark}

\newtheorem{deff}{definition}

\newcommand{\bd}{\begin{deff}}

\newcommand{\ed}{\end{deff}}

\newcommand{\bl}{\begin{lm}}

\newcommand{\el}{\end{lm}}

\newcommand{\bp}{\begin{proo}}

\newcommand{\ep}{\end{proo}}

\newcommand{\bt}{\begin{thee}}

\newcommand{\et}{\end{thee}}

\newcommand{\bc}{\begin{co}}

\newcommand{\ec}{\end{co}}

\newcommand{\brm}{\begin{rem}}

\newcommand{\erm}{\end{rem}}

\usepackage{t1enc}

\newcommand{\newc}{\newcommand}

\let\ccdot.
\def\cdot{\hbox to 2.5pt{\hss$\ccdot$\hss}}

\newc{\aR}{\mbox{\boldmath{$ R$}}}
\newc{\aS}{\mbox{\boldmath{$ S$}}}
\newc{\aT}{\mbox{\boldmath{$ T$}}}
\newc{\aW}{\mbox{\boldmath{$ W$}}}

\newc{\aD}{\mbox{\boldmath{$ D$}}\hspace{-.2mm}}

\newc{\aK}{\mbox{\boldmath{$ K$}}}
\newc{\aL}{\mbox{\boldmath{$ L$}}}

\usepackage{amssymb}
\usepackage{amscd}

\newcommand{\Sc}{\it Sc}

\newcommand{\nn}[1]{(\ref{#1})}

\newcommand{\J}{{\mbox{\it J}}}

\newc{\obstrn}[2]{B^{#1}_{#2}}

\newcommand{\rpl}                         
{\mbox{$
\begin{picture}(12.7,8)(-.5,-1)
\put(0,0.2){$+$}
\put(4.2,2.8){\oval(8,8)[r]}
\end{picture}$}}

\newcommand{\lpl}                         
{\mbox{$
\begin{picture}(12.7,8)(-.5,-1)
\put(2,0.2){$+$}
\put(6.2,2.8){\oval(8,8)[l]}
\end{picture}$}}

\usepackage{ifthen}

\newc{\tensor}[1]{#1}
\newc{\Mvariable}[1]{\mbox{#1}}
\newc{\down}[1]{{}_{#1}}
\newc{\up}[1]{{}^{#1}}

\newc{\JulyStrut}{\rule{0mm}{6mm}}
\newc{\midtenPan}{\mbox{\sf S}}
\newc{\midten}{\mbox{\sf T}}
\newc{\midtenEi}{\mbox{\sf U}}
\newc{\ATen}{\mbox{\sf E}}
\newc{\BTen}{\mbox{\sf F}}
\newc{\CTen}{\mbox{\sf G}}

\def\sideremark#1{\ifvmode\leavevmode\fi\vadjust{\vbox to0pt{\vss
 \hbox to 0pt{\hskip\hsize\hskip1em
 \vbox{\hsize2cm\tiny\raggedright\pretolerance10000
  \noindent #1\hfill}\hss}\vbox to8pt{\vfil}\vss}}}

\numberwithin{equation}{section}

\newcommand{\hh}{{\hspace{.3mm}}}

\newcommand{\cc}{\boldsymbol{c}}

\DeclareMathOperator{\divergence}{div}

\newcommand{\sss}{\scriptscriptstyle}

\newcommand{\SSmidge}{{\hspace{-.2mm}}}

\renewcommand\geq{\geqslant}
\renewcommand\leq{\leqslant}

\DeclareMathOperator{\ext}{d}
\DeclareMathOperator{\D}{L}

\DeclareMathOperator{\I2}{S}

\DeclareMathOperator{\Vol}{Vol}

\newcommand{\mh}{{}\hspace{-.3mm}}

\begin{document}

\renewcommand{\today}{}
\title{
{
Renormalized volumes with boundary
%
}}
\author{ A. Rod Gover${}^\sharp$ \& Andrew Waldron${}^\natural$}

\address{${}^\sharp$Department of Mathematics\\
  The University of Auckland\\
  Private Bag 92019\\
  Auckland 1142\\
  New Zealand,  and\\
  Mathematical Sciences Institute, Australian National University, ACT
  0200, Australia} \email{gover@math.auckland.ac.nz}
  
  \address{${}^{\natural}$Department of Mathematics and Center for Quantum Mathematics and Physics\\
  University of California\\
  Davis, CA95616, USA} \email{wally@math.ucdavis.edu}

\vspace{10pt}

\renewcommand{\arraystretch}{1}

\begin{abstract}
We develop a general regulated volume expansion for 
the volume of a manifold with boundary whose measure is suitably singular along a separating hypersurface.
The expansion is shown to have a regulator independent anomaly  term and a renormalized volume term given by
the  primitive of an associated anomaly operator. These results apply to a wide range
of structures. We detail applications  in the setting of
measures derived from a conformally singular metric. In particular, we show that  the anomaly generates invariant  ($Q$-curvature, transgression)-type pairs
for hypersurfaces with boundary. For the special case of  anomalies coming from the volume enclosed by a minimal hypersurface ending on the boundary of a  Poincar\'e--Einstein
structure, this result recovers Branson's $Q$ curvature and corresponding transgression.
When the singular metric solves a boundary version of the constant scalar curvature Yamabe problem, the anomaly gives generalized Willmore energy functionals  for hypersurfaces with boundary.
Our approach yields computational algorithms for all the above quantities, and we give explicit results for surfaces embedded in 3-manifolds.

\vspace{10cm}


\end{abstract}


\maketitle

\pagestyle{myheadings} \markboth{Gover \& Waldron}{Renormalized Volume with Boundary}

\newpage

\tableofcontents

\section{Introduction}

Many striking results have transpired from the study of infinite
volume manifolds that admit a suitable notion of volume
renormalization. In particular a rich program has emerged by studying
key terms in regulated volume expansions for divergent
volumes. This theory is especially well developed for Poincar\'e--Einstein manifolds. These are
Riemannian manifolds that are (at least asymptotically)
negative Einstein and conformally
compact. The motivation initially physics driven. Following a
dictionary outlined by 
Witten~\cite{Witten}, Gubser, Klebanov, and Polyakov~\cite{Gubser}
for the anti de Sitter/conformal field theory (AdS/CFT) correpondence of~\cite{Mal},
 the  AdS/CFT volume renormalization prescription  
 in this setting was first 
 carried out  by Henningson
and Skenderis in~\cite{Henningson}. For Poincar\'e
Einstein 4-manifolds, Anderson linked the  renormalized
volume to the Chern-Gauss-Bonnet formula~\cite{Anderson}, a theme later studied and
extended to other dimensions in~\cite{Chang}. In~\cite{FGQ}, Fefferman
and Graham proved that for Poincar\'e--Einstein structures  a certain so-called anomaly
term in the volume expansion is a boundary integral over the Branson~$Q$-curvature~\cite{BO,BQ}.

Prior to Anderson's work, Graham and Witten extended the renormalized volume
program to include analogous problems for the divergent ``area'' of
minimal submanifolds that are embedded in conformally compact
geometries and have boundary at a conformal infinity
\cite{GrahamWitten,Gra00}; in the case of minimal 3-manifolds they showed that the
anomaly term gives the Willmore energy functional. A detailed study of
the case of minimal surfaces in Poincar\'e--Einstein 3-manifolds by Alexakis and
Mazzeo showed how the renormalized area is
related to the Willmore functional (now computed for an  embedded minimal
surface).  As in the case of 4-manifolds \cite{Anderson}, and the
advances in~\cite{Chang}, this study involved boundary curvatures
 that arise from the transgression terms in the
Chern--Gau\ss--Bonnet theory. In~\cite{Chang} such terms are in turn
linked the so-called $T$-curvature discovered in \cite{Chang-Qing}.

More recently it was realized that  working with general conformally
compact manifolds leads naturally to an extension to higher
dimensional analogs of the Willmore equation~\cite{CRMouncementCRM,GW15,GW161}, and via volume expansions, energy
functionals for these~\cite{Grahamnew,GWvol}
(see also \cite{GW161,GGHW15} for an alternative approach). Here we show that 
there is tremendous gain in treating a new but related problem, namely
solid infinite volume regions embedded   in manifolds
of the same dimension.  
Asides from allowing the  study of renormalized 
volumes for general regions, 
in particular this provides a direct and
efficient route to the inclusion of boundary/transgression type terms and the generation of higher Willmore energy functionals with boundary.

\medskip 

Our ultimate focus is on structures linked to
conformal geometry. However many aspects of volume renormalization
rely on far less structure and so apply to many other settings and
geometries ({\it cf. e.g.}, \cite{Sheshadri}). With this in mind we consider initially the simple
 setting of  a connected $d$-manifold $\overline{M_+}$ with boundary
$\partial M_+$ (and interior $M_+$)  equipped with a singular volume measure 
$$
{\mu}^o=\frac{{\mu}}{\si^k}\, ,
$$ where $\mu$ is any volume measure on $\overline{M_+}$, $\si$ is
a non-negative defining function for $\partial M_+$
and $k\geq 1$ is a suitable real number.
  Note that
$\mu^o$ is unchanged if we replace the pair $(\mu, \si)$
with $(\Omega^k\mu,\Omega\si)$, where $\Omega$ is any positive
function.
Thus the data $\mu^o$ is equivalent to a density
$\boldsymbol{\si}=[\mu\, ;\, \sigma]$ which is a section of the positive $k^{\rm th}$ root of the
bundle of volume densities. 

The manifold $(M_+,{\mu^o})$ has infinite volume   but via regularization one may consider
the existence of some renormalized volume. In fact we wish to treat
the more general question of the volume of a connected, dimension $d$,
closed submanifold $\overline{D_+}$ (with interior~$D_+$) of~$\overline{M_+}$ that
intersects~$\partial M_+$ along a connected hypersurface $\Sigma= \partial M_+\cap \overline{D_+}$ that forms a closed smoothly bounded 
submanifold of $\partial M_+$. We assume $\partial D_+$ is a smooth
properly embedded
submanifold of $M_+$.
We depict this in the first picture below:
\begin{equation}\label{pictures}
\includegraphics[scale=.4]{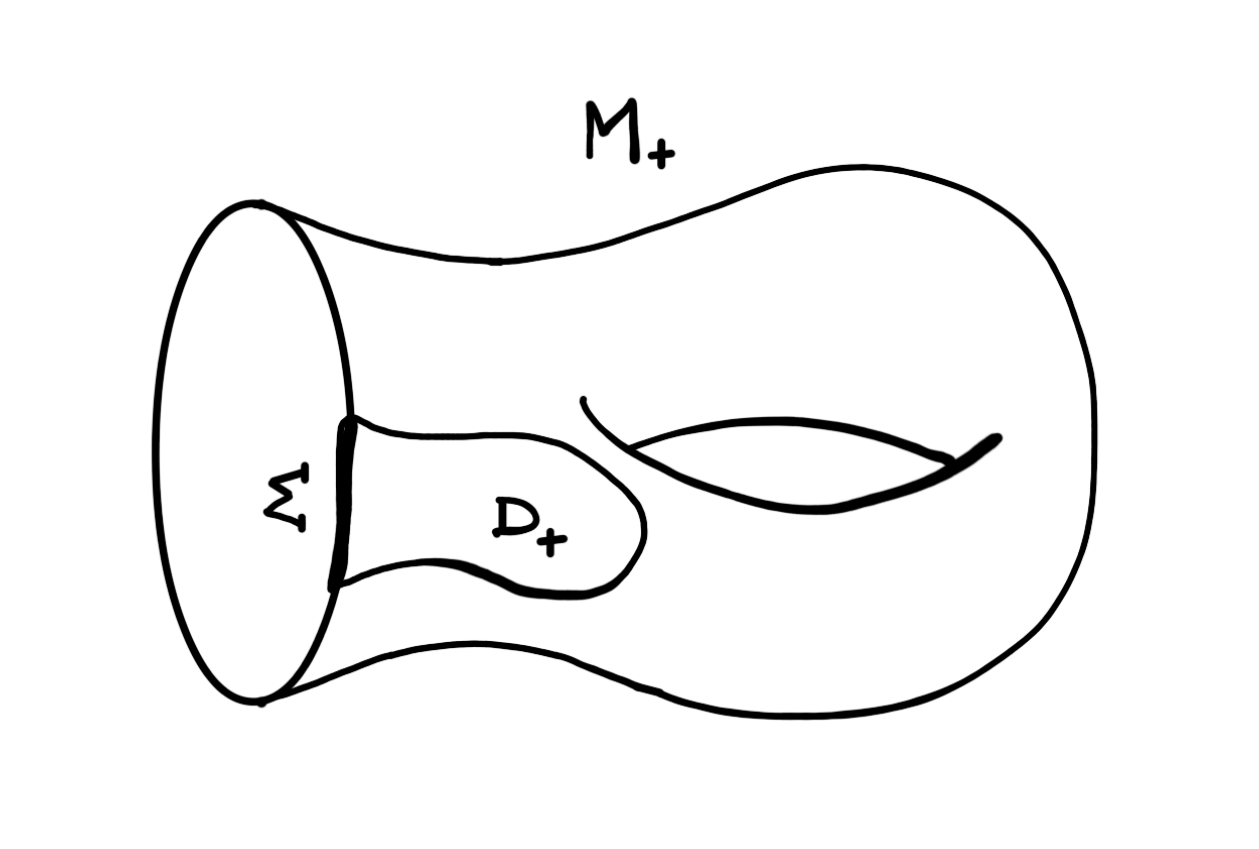} 
\includegraphics[scale=.4]{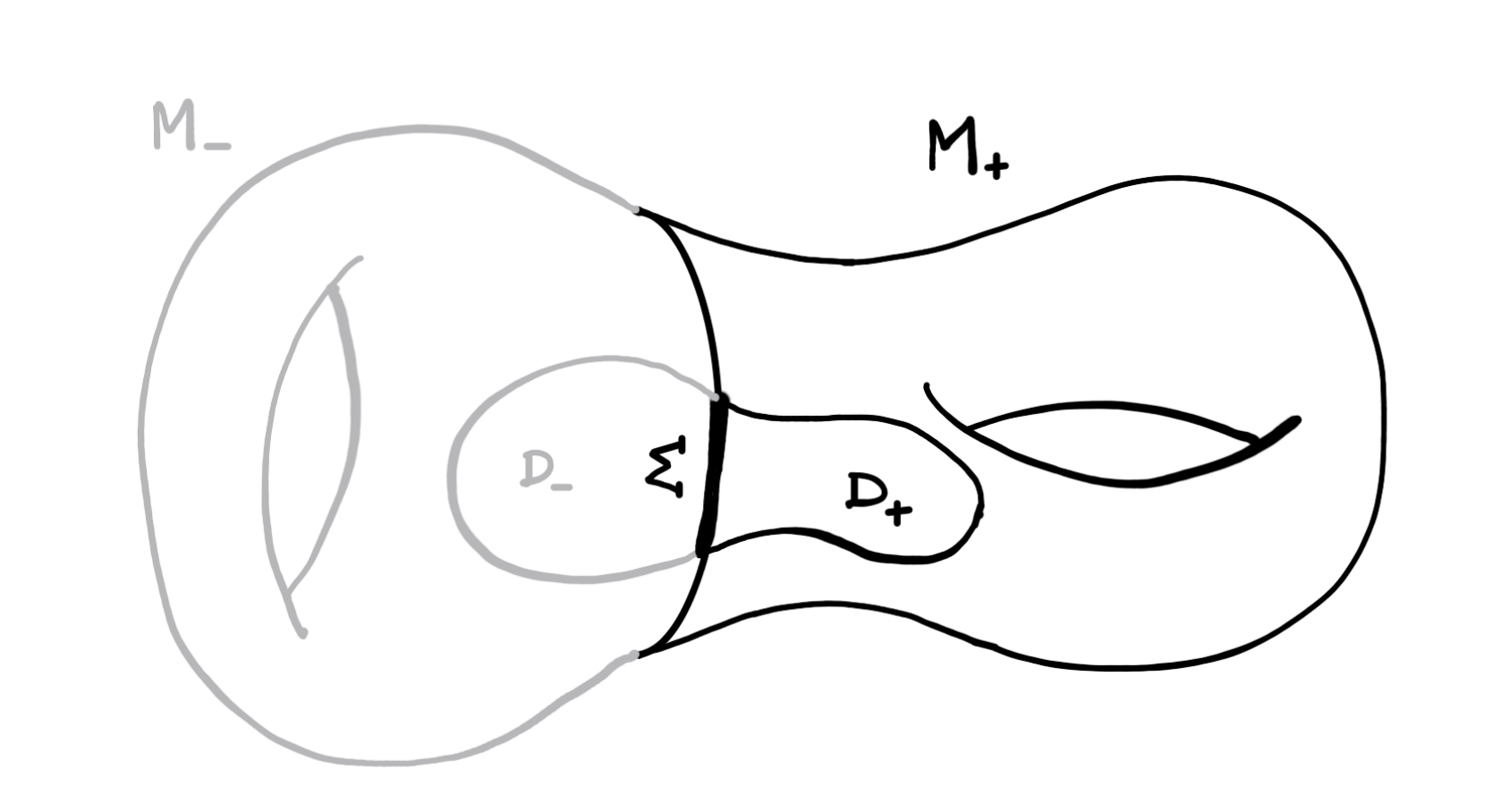}
\end{equation}
With the restriction of $\mu^o$ to ${D_+}$, the structure
$(D_+,\mu^o)$ has infinite volume.
 We treat the 
asymptotics of the regulated volume of $D_\varepsilon:=D_+\cap \{ \si > \varepsilon\}$
in Theorem \ref{expansion} below, and show that if $k\in
\mathbb{Z}_{\geq 1}$ then, as coefficient of $\log \varepsilon$, there is
an {\em anomaly} term $\mathcal{A}$, namely a term that is independent
of the choice of $\si$ representing $\boldsymbol{\si}$. Thus
$\mathcal{A}$ is a property of~$(M,{\mu^o})$ and, in particular, is independent of how the volume computation was regulated.  Via
distributions that will be explained below, the theorem also gives
formulae for the other divergent terms in the expansion. 
The leading term takes the form $A_\Sigma/\big((k-1)\varepsilon^{k-1} \big)
$, where $A_\Sigma$ measures the volume of $\Sigma$ according
to the measure determined by $\mu^o$ and the choice of $\sigma$.
The
renormalized volume is the term in the expansion independent of $\varepsilon$.  Theorem~\ref{transform} shows how this depends on the choice of regulator and also that the latter
is a conformal primitive for an integral anomaly operator, which returns the anomaly when acting on the constant function~$1$.

We next specialize to the conformal setting and this enables a finer
analysis of the structures mentioned. For example we are easily able
to provide an explicit general formula for the two leading divergent
terms in the volume asymototics, see Theorem \ref{divergences}. By
{\em conformal setting} we mean that we assume that $k=d$
and~$\overline{M_+}$ is equipped with an equivalence class $\cc$ of
conformally related Riemannian metrics. So for $g,\widehat{g}\in \cc$
we have that $\widehat{g}=\Omega^2 g$ for some smooth function $\Omega
>0$.  The replacement of $g$ with~$\widehat{g}=\Omega^2 g\in \cc$
determines the replacement of $\mu^g$ with~$\mu^{\widehat{g}}=\Omega^d
\mu^g$, equivalently $\si$ with~$\Omega \si$ (representing~$\boldsymbol{\si}$).

 Each choice of metric $g\in \cc$ induces a metric
$g_\Sigma$ on $\Sigma$, and we show that there is a corresponding pair
of functions $(Q^g,T^g)$, each determined locally 
by the data of $(M,\mu^o)$ and $g$, 
so that
\begin{equation}\label{QT}  
\mathcal{A}=\frac{1}{(d-1)!(d-2)!}\Big[\int_\Sigma Q^g \mu_{g_\Sigma} +\int_{\partial \Sigma} T^g  \mu_{g_{\partial \Sigma}}\Big]\, ,
\end{equation}
see Theorem \ref{Q+T}. The quantities $Q^g$ and $T^g$ arise from a
fixed algorithm used to prove this theorem, which also gives a formula for $Q^g$.
Since the anomaly is determined by~$(M,\mu^o)$, the right-hand-side is
conformally invariant; the same answer $\mathcal{A}$ is obtained if we
compute with $\widehat g:=\Omega^2 g$ and thus $(Q^{\widehat
  g},T^{\widehat g}\hh)$.  So the $T^g$ may be thought of as a  natural {\em
  transgression} for the quantity~$Q^g$ (neither of which is
separately conformally invariant). 

With a suitable restriction ensuring the construction leading to~\nn{QT} depends only on data intrinsic to $\widetilde\Sigma$, we
obtain that $Q^g$ is the usual Branson $Q$-curvature  of $(\Sigma,g_\Sigma)$, and the
corresponding transgression $T^g$  is also determined by this data. In addition  Corollary~\ref{MIN}, which treats this,  shows that for a second canonical restriction sensitive also to the conformal embedding $\widetilde \Sigma\hookrightarrow \overline {M_+}$, the quantity $T^g$  provides a
transgression for the extrinsically coupled $Q$-curvature; the latter
generalizes the $Q$-curvature to a curvature for embedded
hypersurfaces \cite{GWvol}. In view of these cases we shall refer to
any $(Q^g,T^g)$ as a $(Q,T)$-curvature pair.  For
closed hypersurfaces this extrinsically coupled $Q$-curvature provides
a Lagrangian density for higher order Willmore type energies
\cite{Grahamnew,GWvol}.  Thus via $(Q^g,T^g)$ pairs,
where $Q^g$ is the extrinsically coupled $Q$-curvature, the
right-hand-side of \nn{QT} yields a conformally
invariant higher Willmore energy functional for hypersurfaces with
boundary. There is also a nice interpretation of the anomaly
associated with these $Q$-curvature results: it is the anomaly for the
volume expansion of the region enclosed by $\Sigma$ and a minimal
hypersurface in $M$ spanning $\partial \Sigma$. Indeed it is this
setup that we use to obtain the results just mentioned. We are informed
that such geometries have played a role in recent holographic studies of quantum complexity~\cite{complexity1,complexity2}. 

It seems likely that, at least for the case of Branson $Q$-curvature,
the transgression curvature $T^g$ is closely related to the
$T$-curvature for dimension 3 hypersurfaces found  by Chang-Qing~\cite{Chang-Qing}. See also~\cite{Ndi08,Ndi09,Ndi11}
 where this was applied for problems with
$Q$-curvature prescribed on four-manifolds with boundary, and~\cite{GPnew} for 
$T$-like curvatures in other dimensions and orders.

The general theory is illustrated and examined in detail for the case
of surfaces in Sections~\ref{Anomaly} and~\ref{Yamabe}; meaning that we there treat $\Sigma
$ of dimension 2. In particular Theorem~\ref{anomaly} provides, in this
case, general formul\ae\ for $Q^g$ and $T^g$, while Theorem~\ref{singanomaly} specializes this general result to the case depending on embedding $\widetilde \Sigma\hookrightarrow \overline{M_+}$. This gives a Willmore energy expressed
in terms of the Euler
characteristic $\chi_\Sigma$ and explicit
boundary terms generated by an
extrinsic $Q$-curvature/transgression pair.   These results, and the direct
computation algorithms provided by our approach,  are
illustrated by the treatment of several explicit examples, see Sections~\ref{Anomaly} and~\ref{Yamabe}.

%
%

\addtocontents{toc}{\SkipTocEntry}
\section*{Acknowledgements}

We would like to thank Robin Graham, Mukund Rangamani and Erik Tonni for discussions.  Both authors gratefully acknowledge support from the Royal Society of New Zealand via Marsden Grant 13-UOA-018 and A.W. was  supported in part by a Simons Foundation Collaboration Grant for Mathematicians ID 317562.

\section{Results}

Let $M$ be a smooth and (for simplicity) oriented $d$-manifold. Then $M$ is naturally equipped with an equivalence class of measures
 $$\bm \mu = [\mu]=[\Omega^k \mu]\, ,$$
 where $\Omega$ is any smooth, strictly positive function and $k\in {\mathbb R}\backslash\{0\}$. Throughout the paper, we assume that all structures are smooth.
 In the following, the term {\it hypersurface} refers to a smoothly embedded submanifold of codimension~1. 
 
\begin{definition}
Given  an embedded hypersurface $\widetilde \Sigma \hookrightarrow M$, a {\it $k$-defining density} $\bm \sigma$ is an equivalence class of smooth functions
$$
\bm \sigma=[\mu\, ; \sigma]=[\Omega^k\mu \, ;\, \Omega \sigma]
$$
whose zero locus ${\mathcal Z}(\sigma)=\widetilde \Sigma$, with the property that the exterior derivative
$\ext \! \sigma$ satisfies $\ext \! \sigma(P)\neq 0$ $\forall P\in \widetilde \Sigma$ and 
hence $\forall P$
in a neighborhood of $\widetilde \Sigma$.
\end{definition}

\noindent
From now on we suppose that $M$ is equipped with a separating hypersurface~$\widetilde \Sigma$ and we fix $k\in {\mathbb R}\backslash\{0\}$. This is equivalent to a $k$-defining density~$\bm \sigma$.
  The data $(M,\bm \sigma)$ canonically defines a measure
 $$
 \mu^o=\frac{\mu}{\sigma^k}\, ,
 $$
 on $M_+:=\{P\in M| \, \sigma(P)>0\}$.
 We call $\mu^o$ the {\it singular measure} since it diverges along~$\widetilde \Sigma$.

 In the following, we use the term {\it region} 
 for a connected open subset $D\subset M$ with the properties
 (i) $\int_D \mu<\infty$, $ \mu\in \bm \mu$  and $\partial D$ is a closed hypersurface. 
 In particular we wish to study a region $D$ 
that intersects $\widetilde \Sigma$ along some hypersurface $\Sigma\subset \widetilde \Sigma$
that separates $D$ into two regions according to the sign of $\sigma$, as in the second picture of Display~\nn{pictures}, or locally as illustrated:

\vspace{-1cm}
\begin{center}
$\qquad\qquad$\includegraphics[scale=.40]{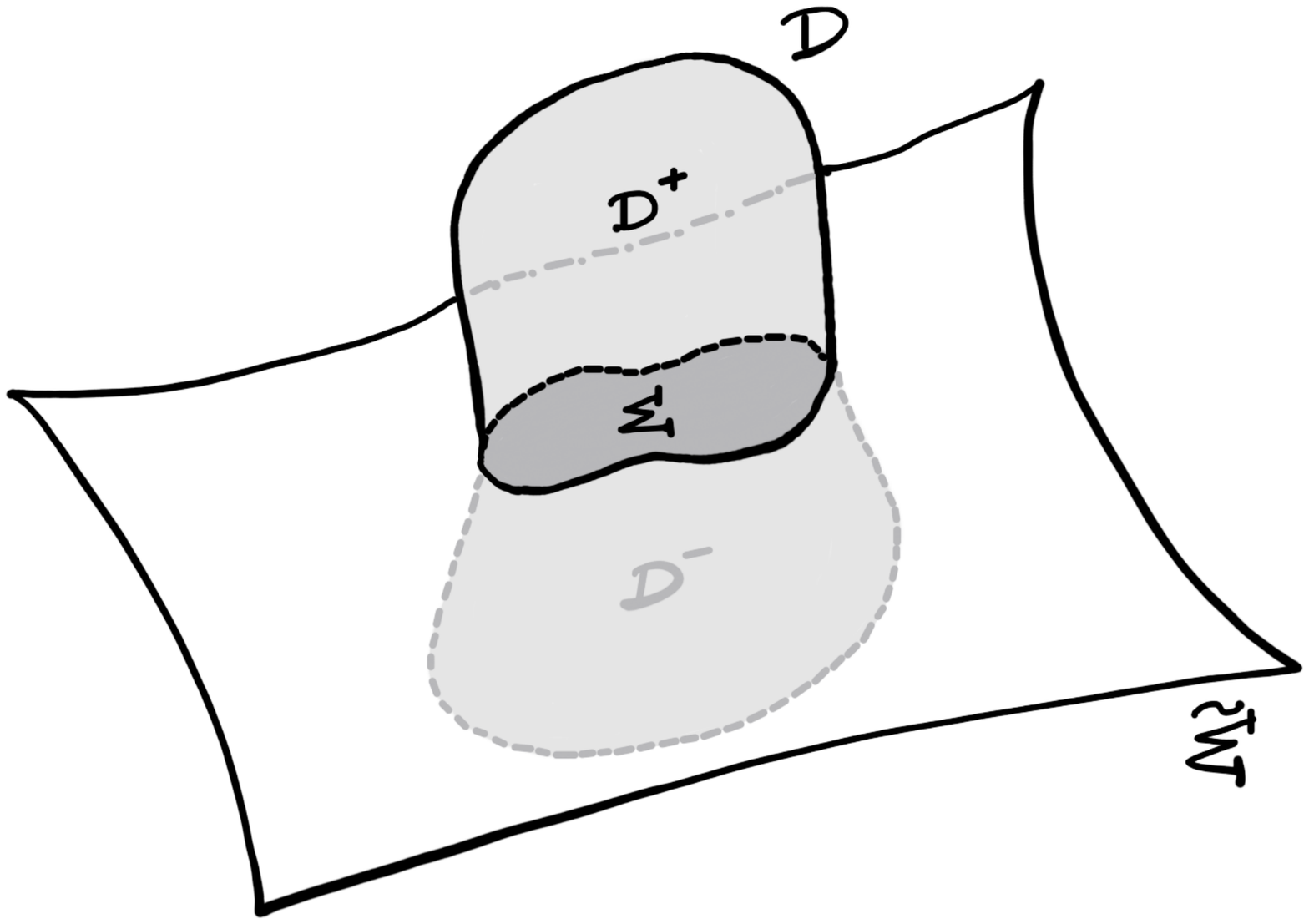}
\end{center}
\vspace{-1cm}

Then the {\it regulated volume}~$\Vol_\varepsilon=\Vol_\varepsilon(D_+;\bm \sigma,\sigma)$ of the region~${D_+}$ is defined by cutting off the integration region at a regulating hypersurface, determined by a choice of $ \sigma\in\bm \sigma$ and $\varepsilon>0$, and inserting a Heaviside step function in the volume integral
\begin{equation}\label{regvoldef}
\Vol_{\varepsilon}:=\int_{ D}{\mu^{o}} \, 
\theta\big({\sigma}-\varepsilon\big)\, ,
\end{equation}
where $\tau$ is any smooth positive function.
The zero locus ${\mathcal Z}(\sigma-\varepsilon)$ defines a one parameter family of {\it regulating hypersurfaces} $\widetilde\Sigma_\varepsilon$ such that $\widetilde\Sigma_0=\widetilde\Sigma$.

\addtocontents{toc}{\SkipTocEntry}\subsection*{The volume expansion}
 
Regulated volume expansions in various specialized geometric  settings (see for example~\cite{Henningson,GrahamWitten,Gra00}) 
are already known to  take the form of a sum of a Laurent series in~$\varepsilon$, whose poles are termed {\it divergences} plus a~$\log\varepsilon$  term whose coefficient is called the {\it anomaly}. Our first theorem establishes this behavior for regulated volumes of general $k$-defining densities:

\begin{theorem}\label{expansion}
Let $\sigma\in \bm \sigma$ where $\bm \sigma$ is a $k$-defining density for a hypersurface $\widetilde \Sigma$. Moreover, let $D$ be a region such that 
\begin{enumerate}[(i)]
\item $D=D_+\cup \Sigma \cup D_-$ where $\Sigma:=D\cap \widetilde \Sigma\neq \emptyset$, 
\\[-3mm]
\item $D_\pm$ are regions such that $\sigma > 0$ on $D_+$ and $\sigma<0$ on $D_-$, and
\\[-3mm]
\item there is  a neighborhood of $\Sigma$ in $D$ with solid cylinder topology.\\[-4mm]
\end{enumerate}
Then, given $\varepsilon \in {\mathbb R}_{>0}$
and
$k\in {\mathbb Z}_{\geq 1}$,
the regulated volume
\begin{equation*}
\Vol_\varepsilon(D_+;\bm\sigma,\sigma)=
\sum_{\ell\in \{k-1,\ldots,1\}}
\frac{(-1)^{k-\ell-1} \, V_{\ell}}{(k-\ell-1)!\, \ell}\, \frac{1}{\varepsilon^\ell}\, 
 \, +\, 
{\mathcal A}\, \log\varepsilon
\, + \, 
\Vol_{\rm ren}
\, +\, \varepsilon\, {\mathcal R}(\varepsilon) \, ,
\end{equation*}
with
\begin{eqnarray*}
V_{\ell}&=&\int_{D} \mu\, 
\delta^{(k-\ell-1)}(\sigma)
\ = \ \int_\Sigma v_{\ell}(\sigma)\, 
\, ,\\[3mm]
{\mathcal A}&=&\frac{(-1)^{k-1}}{(k-1)!}\, \int_{D} \mu \, 
\delta^{(k-1)}(\sigma)\ = \ \int_\Sigma a(\bm\sigma)\, .
\end{eqnarray*}
Here $\bm \sigma = [\mu\, ;\, \sigma]$ and $\delta^{(n)}$ denotes $n$ derivatives of the Dirac delta distribution. Moreover,~$v_\ell$ and $a$ are functions multiplied by  a measure
on $\Sigma$, both of which are  determined locally (meaning finitely many jets) by the data indicated,
$\Vol_{ren}$ is independent of $\varepsilon$ and the function ${\mathcal R}(\varepsilon)$ depends smoothly on~$\varepsilon$.
When $k=1$, the divergences given by the sum over $\ell$ are absent, while when $k>1$ but not integer the $\log \varepsilon$ anomaly term is absent and the sum over $\ell$ runs over $\{k-1,\ldots,k-\lfloor k\rfloor\}$. 
\end{theorem}

\noindent 
The proof of this theorem is given in Section~\ref{volexps}.

\addtocontents{toc}{\SkipTocEntry}\subsection*{Densities}
The analysis of the volume expansion is simplified using densities:  Here a {\it density of homogeneity~$s$} refers to 
 a double equivalence class  $[\mu\, ;\, \gamma]=[\Omega^k \mu\, ;\, \Omega^s \gamma]$ for~$\gamma$ a section of some vector bundle over $M$ and $s\in {\mathbb R}$. A scalar-valued density $\bm \tau$ of homogeneity~$s=1$ is called a {\it true scale} (or simply a scale).

We may employ a true scale $\bm \tau$ to record the dependence of the regulated volume and related quantities on the choice of $\sigma\in \bm \sigma$. This is done by choosing an equivalence class representative
$\bm \tau=[\mu\, ;\, 1]$ where $\mu$ is the measure determined by the choice of $\sigma$ according to $\bm \sigma=[\mu\, ;\, \sigma]$.
Since the choice $\sigma$ determines a 
 family of regulating surfaces~$\widetilde\Sigma_{\varepsilon}$, we shall call the density $\bm \tau$ a {\it regulator}. We shall therefore employ the notation $\Vol_\varepsilon(D_+;\bm \sigma,\bm \tau)$ for the regulated volume as defined in Equation~\nn{regvoldef}.

 It is also useful to introduce {\it weight~$l$ log densities}~\cite{GW},
 which are defined by an additive equivalence class of functions
 $$
 \bm \lambda=[\mu\, ;\, \lambda]=[\Omega^k\mu\, ;\, 
 \lambda + l \log \Omega]\, .
 $$ 
The logarithm of a true scale $\bm \tau=[\mu\, ;\, \tau]$  defines a weight~$1$ log density $\log \bm \tau :=[\mu\, ;\, \log \tau]$.

The integral of a homogeneity $s=-\dim M$, scalar-valued density $\bm u=[\mu\, ;\, u]=[\Omega^k \mu\, ; \, \Omega^{-k}\, u]$  over a region $D\subset M$  is well-defined and given by
$$
\int_D \bm u := \int_D \mu u\, .
$$
A consequence of the property $\delta(\alpha x)=\tfrac1\alpha \delta(x)$, is that given $\bm \sigma=[\mu\, ; \, \sigma]$ we can form
$$
\bm \delta^{(\ell)}:=\big[\mu\, ;\, 
\delta^{(\ell)}(\sigma)\big]
=
\big[\Omega^k\mu\, ;\, \Omega^{-\ell-1}\, 
\delta^{(\ell)}(\sigma)\big]
\, ,
$$
so $\bm\delta^{(\ell)}$ is a distribution-valued density and $\bm\delta^{(k-1)}$ is a homogeneity~$-\dim M$ density and so can  be integrated over a region $D$. In particular, the result for the regulator independent
anomaly of  Theorem~\ref{expansion}
can be restated as
\begin{equation}\label{A}
{\mathcal A}=
\frac{(-1)^k}{(k-1)!}\, \int_{\,  D}
\, \bm \delta^{(k-1)}\, ,
\end{equation}
which only depends on $(D_+,\bm \sigma)$.

\addtocontents{toc}{\SkipTocEntry}\subsection*{Renormalized volume}
The~$\varepsilon$ independent 
term~$\Vol_{\rm ren}$ in  Theorem~\ref{expansion}  is known as the {\it renormalized volume},
and in general depends on the regulator $\bm \tau$.
Our second theorem 
gives an explicit formula for the regulator dependence of the renormalized volume.

\begin{theorem}\label{transform}
Let $(D,\bm \sigma)$ obey the same conditions as required by Theorem~\ref{expansion}, and 
suppose $\omega$ is
any smooth function. Then the renormalized volume satisfies
$$
\Vol_{\rm ren}(e^{\omega}\bm \tau)-\Vol_{\rm ren}(\bm \tau)=\frac{(-1)^{k}}{(k-1)!}\, \int_D  \omega \, \bm  \delta^{(k-1)}(\sigma)=\int_\Sigma a(\bm \sigma, \omega)\, ,
$$
where $a$ is a locally determined, homogeneity $-\dim\Sigma$, density along $\Sigma$. 
\end{theorem}
\noindent
The proof is given in Section~\ref{volexps}.
In the above, the product of a function $f$ and a density $\bm \gamma=[\mu\, ;\, \gamma]$ is defined by $f \bm \gamma:=[\mu\, ;\, f\gamma]$. When $\omega$ is a constant function, the result above is proportional to the anomaly. The local function $a$ determines the {\it anomaly operator}~$\int_\Sigma a(\bm \sigma,\raisebox{.8mm}{\scalebox{1.2}{.}}\, )$.

\addtocontents{toc}{\SkipTocEntry}\subsection*{The conformal setting}

Our remaining results concern the conformal setting $\bm \mu = [\mu^g]$ (where $\cc=[g]=[\Omega^2g]$, $k=d$), equipped with a defining density $\bm \sigma=[g\, ;\, \sigma]$. 
Here and in general, 
   a density $\bm \gamma=[\mu^g\, ;\, \gamma]$ is now  denoted by~$[g\, ;\, \gamma]$ and termed a
   {\it conformal density of weight~$w$} when the homogeneity $s=w$ (so that
$[g\, ;\, \gamma]=[\Omega^2 g\, ; \Omega^w \gamma]$). 
Note now, that a choice of true scale~$\bm \tau$ coincides with a choice of $g\in\cc$.
The distinction between the notion of conformal weight and homogeneity is made because  other geometric structures have natural notions of weight that may not coincide with the homogeneity.

The geometry of the region $D$ and hypersurface~$\widetilde \Sigma$ is that described above, but now we are 
 studying  the regulated volume of the region~${D_+}$ with respect to a singular metric
$$
g^o=\frac{g}{\sigma^2}\, .
$$
Since~$g^o=(\Omega^2 g)/(\Omega \sigma)$, this and the corresponding singular volume form~${\mu^{g^o}}$ only depend on~$\bm \sigma$.
Also, the
hypersurface~$\widetilde\Sigma$  given by the zero locus~${\mathcal Z}(\sigma)$ is often called a {\it conformal infinity} of the metric~$g^o$.

The volume expansion result of Theorem~\ref{expansion} in terms of integrals over 
distribution-valued densities adapted to the conformal setting is given below (this result was first given in~\cite[Theorem 3.1]{GWvol})
\begin{multline}
\label{epsilon_expansion}
\Vol_\varepsilon  \ =\ 
\frac1{d-1}\frac1{\varepsilon^{d-1}}\int_{\,  D}
\frac{\bm \delta}{{\bm \tau}^{d-1}}
\ -\ \frac{1}{d-2}\frac1{\varepsilon^{d-2}}
\int_{\,  D}
\frac{\bm \delta'}{{\bm \tau}^{d-2}}
\ +\ 
\frac{1}{2(d-3)}\frac1{\varepsilon^{d-3}}
\int_{\,  D}
\frac{\bm \delta''}{{\bm \tau}^{d-3}}\ +\ 
\cdots
\\[3mm]\qquad 
\ \  \cdots \ +\ 
\frac{\ (-1)^{d-2}\ }{(d-2)!}\frac1{\varepsilon}
\int_{\,  D}
\frac{\bm \delta^{(d-2)}}{{\bm \tau}}
\ + \
\frac{(-1)^d}{(d-1)!}\,\log\varepsilon\  \int_{\,  D}
\, \bm \delta^{(d-1)}
\ +\ \Vol_{\rm ren}\ +\  \varepsilon\,  {\mathcal R}\, (\varepsilon)\, .
\end{multline}
Integrations~$\int_\bullet\,~$ over regions in~$M$
are  defined above; when~$\bullet$ is a hypersurface or embedded submanifold, the integration measure is that of the corresponding conformal class of induced metrics.

\addtocontents{toc}{\SkipTocEntry}\subsection*{The anomaly}

The next theorem expresses the anomaly as a sum of integrals over $\Sigma$ and $\partial \Sigma$.  Although the anomaly does not depend on the regulator, simple formul\ae\  for the respective integrands are obtained by introducing a new true scale~$\bm \tau$, 
on which the final result also does not depend. 
It is thus not necessary that this true scale coincide with the regulator, but it is often simplifying to make this choice.

\begin{theorem}\label{Q+T}
Let $(M,\cc)$ be a conformal manifold.
For any choice of true scale $\bm \tau$, the 
anomaly in the regulated volume $\Vol_\varepsilon(D_+;\bm \sigma,\bm \tau)$ is given by
$$
{\mathcal A}=
\frac{1}{(d-1)!(d-2)!}\,\left[ \ 
\int_\Sigma 
\bm Q^{\scalebox{.7}{$\bm\sigma$}}_{\scalebox{.7}{$\bm\tau$}}
\, +\, 
\int_{\partial \Sigma} \bm T^{\scalebox{.7}{$\bm\sigma$}}_{\scalebox{.7}{$\bm\tau$}}\ \right]\, ,
$$
where
$$
\bm Q^{\scalebox{.7}{$\bm\sigma$}}_{\scalebox{.7}{$\bm\tau$}}=
\frac1{\sqrt{\bm{\mathcal S}}}\, (-\D\circ\I2^{-1})^{d-2}\circ  
(-{\mathcal L})
\,  \log {\bm \tau}\Big|_\Sigma\, ,\qquad
{\mathcal L}=\I2^{-1}\circ \D  -
(\bm \nabla_a \bm {\mathcal S}^{-1}) \, {\bm g}^{ab}
  \triangledown^{\scalebox{.7}{$\bm\sigma$}}_b\, ,
$$
and $\bm T^{\scalebox{.7}{$\bm\sigma$}}_{\scalebox{.7}{$\bm\tau$}}$ has a local formula in terms of $(\bm \sigma,\bm \tau)$.
\end{theorem}

\noindent
In the above Theorem, 
$\bm \nabla$ and $\triangledown^{\scalebox{1.3}{$.$}}$
are, respectively, the conformal gradient (exterior derivative) and coupled conformal gradient operators defined in Equations~(\ref{conformalgradient},\ref{triangle}), and~$\D$ denotes the {\it Laplace--Robin operator}~\cite{GoSigma} (see also~\cite{GW,GWvol}) determined by the defining density $\bm \sigma=[g\, ;\, \sigma]$. The latter is a  second order differential operator mapping a scalar-valued weight $w$ density~$\bm \varphi$ to a scalar-valued weight $w-1$ density, and a weight~$w$ log density $\bm \lambda$ to a weight~$-1$ conformal density, respectively according to
\begin{eqnarray}\label{LR}
\D\bm\varphi\ =\: \D [g\, ;\, \varphi]&:=&\big[g\, ;\, 
(d+2w-2) (\nabla_n+w \rho) \varphi
-\tfrac 1 d\,  \sigma \, (\Delta^g  + w \J^{\, g} )\varphi 
\big]\, ,\\[2mm]
\label{logact}
\D \bm\lambda\ =\: \D [g\, ;\, \lambda] &:=& [g\, ;\, (d-2)(\nabla_n\lambda +w \rho) - \sigma(\Delta^g \lambda + w \J\,^g )] \, .
\end{eqnarray}
Here $n:=\nabla \sigma$, $\rho:=-\tfrac 1d (\nabla.n + \J^{\,  g} \sigma)$ and $\J^{\, g}$ is related to the scalar curvature of $g$  by $\J^{\, g}=\tfrac{1}{2(d-1)}\, \Sc^g$. We use $\nabla^g$ (or simply $\nabla$ when clear by context)  to denote the Levi-Civita connection of $g$ and $\Delta^g$ is its negative energy Laplacian, and we employ a dot notation for the divergence. In our conventions, the  scalar curvature $\Sc^g$  of $\nabla^g$ is negative for hyperbolic spaces.
\vspace{1mm}

Along the zero locus~${\widetilde\Sigma}$ of $\bm \sigma$, provided $w\neq 1-\tfrac d2$, the 
Laplace--Robin operator gives the conformally invariant, Robin-type, combination
of Neumann and Dirichlet boundary operators of~\cite{cherrier}, while in the interior  it is  an invariant,  scale-coupled extension of the Laplacian.
Also,
\begin{equation}\label{scurvy}
{\bm {\mathcal S}}:=\big [g\, ;\, |n|_g^2+2\rho\sigma\big]
\end{equation}
is a weight~$w=0$ conformal density known as the {\it ${\mathcal S}$-curvature}. Along ${\widetilde\Sigma}$ 
it measures the squared length of the conormal vector $n$, while in the interior it is an invariant, scale-coupled,  extension of the scalar curvature up to a negative constant. The operator~$\I2$ denotes multiplication by $\bm{\mathcal S}$, and by inspection, is clearly invertible in a neighborhood of $\widetilde\Sigma$ (we lose no generality assuming this holds throughout $D$).

In general the above theorem provides a $(Q^g,T^g)$ pair determined by the defining density~$\bm \sigma$ and region $D$. 
It is interesting to ask whether this construction can be used to
give a $(Q^g,T^g)$-pair determined only by 
the conformal embedding $\widetilde\Sigma\hookrightarrow M$ and the choice of $\partial \Sigma\hookrightarrow \widetilde \Sigma$.
\begin{corollary}[of Theorem~\ref{Q+T}]
\label{MIN}
There exists natural formul\ae\ for the pair $(\bm Q^{\scalebox{.7}{$\bm\sigma$}}_{\scalebox{.7}{$\bm\tau$}},
 \bm T^{\scalebox{.7}{$\bm\sigma$}}_{\scalebox{.7}{$\bm\tau$}})$ of Theorem~\ref{Q+T}  determined canonically by the conformal embeddings as follows:
 \begin{enumerate}[(i)]
 \item $\partial\Sigma\hookrightarrow \widetilde\Sigma\hookrightarrow \overline{M_+}$  
 and 
 $\bm Q^{\scalebox{.7}{$\bm\sigma$}}_{\scalebox{.7}{$\bm\tau$}}$
is the extrinsic $Q$-curvature determined by $\widetilde\Sigma\hookrightarrow \overline{M_+}$.
   \\[-2mm]
 \item  $\partial \Sigma\hookrightarrow \widetilde\Sigma$
 and
 $\bm Q^{\scalebox{.7}{$\bm\sigma$}}_{\scalebox{.7}{$\bm\tau$}}$
is the Branson $Q$-curvature of $(\widetilde \Sigma,\cc_{\hh \widetilde \Sigma})$.
  \end{enumerate}
 
\end{corollary}

\noindent
In (i) above the extrinsic $Q$-curvature is
canonically determined by solving the singular Yamabe condition~\nn{singYam}, see~\cite{GW15,GWvol}.
On the other hand, in (ii) $\cc_{\hh \widetilde \Sigma}$ is 
any conformal structure on $\widetilde \Sigma$. 
In both cases, we may view $\bm \tau$ as any true scale on $\widetilde \Sigma$.
The proofs of Theorem~\ref{Q+T} and its Corollary~\ref{MIN} are given in Section~\ref{QTsect}.

\begin{remark}
For Poincar\'e--Einstein structures, where $\partial \Sigma=\emptyset$, the anomaly was already known to be the integral of the Branson $Q$-curvature~\cite{FGQ}. 
\end{remark}

\addtocontents{toc}{\SkipTocEntry}\subsection*{Leading divergences}

Our next result 
 gives explicit formul\ae\ for the leading order~$1/\varepsilon^{d-1}$ and next to leading (nlo)~$1/\varepsilon^{d-2}$ divergences in  the regulated volume expansion of Equation~\nn{regvoldef}. 
 The contribution to the nlo divergence from the   boundary $\partial \Sigma$, is proportional to the intersection angle between the surfaces $\Sigma$ and $\partial  D$.
This is invariantly conformally defined as follows: Given a defining function $\nu$ for some hypersurface~$\Lambda$, we have  $$\frac{\nabla \nu\, }{|\nabla \nu|_g}\, 
\stackrel{\Lambda}=\, 
\frac{\nabla (\Omega\nu)}{\,\,  \Omega|\nabla (\Omega\nu)|_{\Omega^2 g}}\, ,
$$
where $\stackrel{\Lambda}=$ denotes equality upon restriction to a hypersurface. Hence, $$\bm{\hat n}_\Lambda:=[g\, ;\ \nabla \nu/|\nabla\nu|_g]$$
defines a weight $w=1$ conformal density on the conformal manifold $({\Lambda},\bm { c}_\Lambda)$ where ${\bm c}_\Lambda=[g_\Lambda]$ is the conformal class of metrics containing the metric $g_\Lambda$ induced from $g$. This density is termed the {\it unit conormal}. Thus, the angle $\theta$ between ${\Sigma}$ and $\partial  D$ is defined, as a function $\partial \Sigma \to \big [\!-\tfrac\pi 2,\tfrac\pi 2\big )$, by the relation
$$
\cos \theta = \bm{\hat n}_\Sigma . \bm{\hat n}_{\partial  D}\, . 
$$
The leading and nlo divergences are regulator dependent and given as follows:

\begin{theorem}\label{divergences}
Let~$(M,\cc)$ be a conformal $d$ manifold with $d>2$. Then the leading and nlo divergences in the regulated
volume expansion of Theorem~\ref{expansion}
are given by 
\begin{equation}\label{boundaryleading}
\begin{split}
\Vol_\varepsilon\ &=\ 
\frac{1}{d-1}
\frac
1{\varepsilon^{d-1}}
\int_{\Sigma} \frac{1}{\sqrt{\bm{\mathcal S}}\bm \tau^{d-1}}
\\[2mm]
&\ -
\frac{1}{d-2}
\frac{1}{\varepsilon^{d-2}}
\left(
\frac1{d-2}
\int_{\Sigma}
 \frac1{\sqrt{\bm{\mathcal S}}}\,  \D \Big(\frac{1}{\bm {\mathcal S}  \bm \tau^{d-2} }\Big)
+
\int_{\partial \Sigma} \frac{\cos\theta}{\sqrt{
\bm{  \mathcal S}\hspace{.1mm}
\bm {\mathcal S}_{\scalebox{.7}{$\partial D$}}^{\phantom{.}} \,} \ \bm \tau^{d-2}}
\right)\ + \cdots\, .
\end{split}
\end{equation}
\end{theorem}

\medskip

\addtocontents{toc}{\SkipTocEntry}\subsection*{Surfaces}

Our remaining results focus on the anomaly and its associated $(Q,T)$-pair for embedded
surfaces.

\begin{theorem}
\label{anomaly}

Let $(M,\bm c)$ be a conformal $3$-manifold.
Then the anomaly of Theorem~\ref{expansion} (see also Equation~\nn{A})
is given by
\begin{equation}\label{surfaceanomaly}
{\mathcal A}=\frac12\int_{\Sigma}
{\bm Q}^{\scalebox{.7}{$\bm\sigma$}}
+\frac12
\int_{\partial\Sigma}
{\bm T}^{\scalebox{.7}{$\bm\sigma$}}\, ,
\end{equation}
with
\begin{equation*}
\begin{split}
{\bm Q}^{\scalebox{.7}{$\bm\sigma$}}
&=
\frac1{{\sqrt{\bm{\mathcal S}}}} \D\circ\I2^{-1}\circ\, 
{\mathcal L}\, 
 \log \bm \tau\, ,
 \\[3mm]
 {\bm T}^{\scalebox{.7}{$\bm\sigma$}}&= 
\frac{
\cos \theta }{\sqrt{
\bm{  \mathcal S}\hspace{.1mm}
\bm {\mathcal S}_{\scalebox{.7}{$\!\partial D$}}^{\phantom{.}} \,} } \ 
{\mathcal L}
\,   \log \bm \tau
\, + \, 
 \bm{\hat n}_{\partial\Sigma|\partial  D}^a\bm {\nabla}^{\partial D}_a
\Big(\frac{\bm{\hat n}_{\scalebox{.7}{$\partial D$}} \, .\, 
\bm n^{\scalebox{.7}{${\bm\tau}$}}
}{\bm {\mathcal S}\hspace{.1mm}\bm {\mathcal S}_{\partial D}}\Big)\,  ,
\end{split}
\end{equation*}
where
$$
  {\mathcal L}=\I2^{-1}\circ \D  -
(\bm \nabla_a \bm {\mathcal S}^{-1}) \, {\bm g}^{ab}
  \triangledown^{\scalebox{.7}{$\bm\sigma$}}_b\, \mbox{ and }\ 
  \bm n^{\scalebox{.7}{${\bm\tau}$}}
=\bm \tau^{-1} 
 \triangledown^{\scalebox{.6}{$\bm \tau$}}
 \bm \sigma\, .
$$
Here~$\bm {\mathcal S}_{\partial D}$ denotes the~${\mathcal S}$-curvature along~$\partial  D$
of the restriction of~$\bm \sigma$ to~$\partial  D$, and 
~$\bm{\hat n}_{\partial\Sigma|\partial  D}$ is the unit conormal with respect to the embedding of~$\partial\Sigma$ in~$\partial  D$. Also,
$\cos \theta = \bm{\hat n}_\Sigma . \bm{\hat n}_{\partial  D}$ and~$\bm{ \nabla}^{\partial D}$ is the conformal gradient along~$\partial D$.
\end{theorem}
\noindent
The proofs of Theorems~\ref{divergences} and~\ref{anomaly} are given, respectively,  in Sections~\ref{Divergences} and~\ref{Anomaly}. In the latter  of those sections, we also given an 
explicit class of examples for which $\Vol_{\varepsilon}$ can be computed exactly and exhibit how these two theorems generate the corresponding divergences and anomaly.

\medskip

We also want to  compute anomalies when 
 the singular metric~$g^o$  obeys the {\it singular Yamabe
condition} on its scalar curvature:
\begin{equation}\label{singYam}
\Sc^{g^o}=-d(d-1)\big(1 + \sigma^d B\big)\, ,\mbox{ where } B\in C^\infty\!M\, .
\end{equation}
This is of particular interest because, as shown in~\cite{ACF}, the conformal embedding ${\widetilde\Sigma}\hookrightarrow M$ determines~$\sigma$ uniquely up to the addition of terms of order  $\sigma^{d+1}$. Furthermore, the singular Yamabe condition~\nn{singYam} can be reformulated as a  unit condition for the ${\mathcal S}
$-curvature~\cite{Goal}:
\begin{equation}\label{one}
\bm{\mathcal S} =1+{\mathcal O}(\bm \sigma^d)\, .
\end{equation}
This may  be viewed as a condition on the defining density $\bm \sigma$, and solutions are termed {\it conformal unit defining 
densities}. 
This allows the 
boundary calculus for conformally compact
  manifolds of~\cite{GW}
  to be applied to the singular Yamabe problem.
Indeed, based on those methods, 
 it was shown in~\cite{CRMouncementCRM,GW15,GW161} that uniqueness properties of the solution to the singular Yamabe condition can be exploited to study the conformal geometry of embedded hypersurfaces.
We shall show that in the singular Yamabe setting, the anomaly for embedded surfaces 
 can be expressed in terms of the Euler characteristic~$\chi_{\Sigma}$ of the boundary
and {\it conformal hypersurface invariants} integrated over~$\Sigma$ and~$\partial\Sigma$.

The notion of a conformal hypersurface invariant is  defined and studied in detail in~\cite{GW15,GW161} (see also~\cite{Stafford,YuriThesis}) and refers to invariants of a hypersurface (or submanifold)~${\Sigma}$ determined by the conformal embedding of~${\Sigma}\hookrightarrow M$, in particular if~$P({\Sigma},g)$ is a hypersurface invariant determined by the embedding of~${\Sigma}\hookrightarrow (M,g)$ and~$P({\Sigma},\Omega^2 g):=\Omega^w P({\Sigma},g)$, we shall denote the corresponding weight~$w$ conformal hypersurface invariant by~$\bm P=[g\, ; \, P({\Sigma},g)]=[\Omega^2 g\, ;\, \Omega^w P({\Sigma}\, ; g)]$.   Important examples are the unit conormal to a hypersurface~$\bm {\hat n}^{{\Sigma}}:=[g\, ; \ext\! \sigma/|\! \ext\! \sigma|_g]$ and the trace-free part~$\IIo_{ab}$ of the second fundamental form of~$\Sigma$ yielding~$\hh \bm \IIo_{ab}^\Sigma:=[g\, ;\, \IIo_{ab}]$; these both have  weight~$w=1$;  we will drop the label~${\Sigma}$ when the underlying hypersurface is clear from context.
Note that we  often use the same abstract index notation for both hypersurface and ambient tensors, a dot to indicate inner products of vectors while an index replaced by a vector denotes either the canonical contraction of a vector and a tensor, or contraction using the metric or {\it conformal metric}~${\bm g}_{ab}:=[g\, ; g_{ab}]$. 

Having established these notions, we can state our result for the anomaly when~$g^o$ obeys the singular Yamabe condition~\nn{singYam}:

\begin{theorem}\label{singanomaly}
Let $(M,\bm c)$ be a conformal $3$-manifold
and $\bm \sigma$ be a conformal unit defining density,
and let~$\bm {\hat q}$ and~$\bm {\hat p}$
respectively denote the unit conormals of the conformal embeddings~$\partial\Sigma \hookrightarrow \partial  D$ and
$\partial\Sigma \hookrightarrow \partial \widetilde \Sigma$. Then,
 assuming~$\theta$ is nowhere vanishing, the anomaly of the regulated volume is given by
$$
{\mathcal A} = \pi \chi_{\Sigma}
-\frac14\int_{\Sigma} \bm\IIo_{ab}^\Sigma\,  \bm \IIo_\Sigma^{ab}+
\int_{\partial\Sigma} \Big(
\frac{\, \bm \IIo^{\partial D}_{\scalebox{.7}{$\bm{\hat q}\bm {\hat q}$}}-\cos\theta\, \bm \IIo^\Sigma_{\scalebox{.7}{$\bm{\hat p}\bm {\hat p}$}}}{\sin^3\theta}
-\frac12 \cot\theta\,  \bm \IIo^\Sigma_{\scalebox{.7}{$\bm{\hat p}\bm {\hat p}$}}
\Big)
\, .
$$
\end{theorem}

\noindent
The proof of this theorem is given in Section~\ref{Yamabe}, where we also give formul\ae\ for the divergences and an example in which we solve the singular Yamabe condition and compute the corresponding anomaly.
The non-vanishing assumption  on the angle~$\theta$ between the hypersurfaces~${\Sigma}$ and~$\partial D$ in the above theorem is
imposed because of the solid cylinder assumption of Theorem~\ref{expansion}.

\section{Volume expansions}\label{volexps}

In this section we prove Theorems~\ref{expansion} and ~\ref{transform}, 
as well as providing a simple example of the latter result.
The proof of Theorem~\ref{expansion} follows closely the proof of Theorem~3.1 of~\cite{GWvol},
and relies on the same distributional identities employed there (see in particular Section~2.8 of that work). In particular, note that the solid cylinder topology assumption of this theorem is imposed as technical requirement to facilitate
an analysis based on simple distributional identities. Further analysis, may lead to a wider range of applicability of the result.

\begin{proof}[Proof of Theorem~\ref{expansion}]
The main task is to establish the  quoted volume expansion.
First, it is easy to verify that the leading behavior of the regulated volume in the small $\varepsilon$ limit behaves like~$\varepsilon^{1-k}$. Hence we consider the quantity
$$
-\varepsilon^k\frac{d\Vol_{\varepsilon}}{d\varepsilon}=\varepsilon^k
\int_D\frac{\mu}{\sigma^k}\, \delta\big(\sigma-\varepsilon\big)
=\int_D\mu\, \delta\big(\sigma-\varepsilon\big)=: I(\varepsilon)
\, .
$$
Here we have employed the distributional identities $d\theta(x)/dx=\delta(x)$ and~$\delta(\alpha x)=\tfrac1\alpha\delta(x)$.
This is justified by the solid cylinder assumption which allows us to employ $\sigma$ as a coordinate on $D$ in a neighborhood of $\Sigma$.
Upon performing the Dirac delta function integration, the quantity $I(\varepsilon)$ is given by the integral of a positive smooth measure over 
the regulating hypersurface $\Sigma_{\varepsilon}$, and is therefore a smooth function of $\varepsilon$.
Therefore, we may expand $I(\varepsilon)$ as a Taylor polynomial plus remainder about $\varepsilon=0$. The coefficients are therefore given by integrals over differentiated delta functions of $\sigma$. We then integrate the relation $-\varepsilon^k\frac{d\Vol_{\varepsilon}}{d\varepsilon}=I(\varepsilon)$ to obtain the quoted result for the regulated volume~$\Vol_\varepsilon$. The renormalized volume~$\Vol_{\rm ren}$ is the undetermined constant term in this integration. The locality statement for the coefficients of the divergences $V_{\ell}$ and the anomaly ${\mathcal A}$ is simply a statement of the meaning of integrals over  differentiated delta functions. 
The independence of the anomaly measure $a(\bm \sigma)$ of the choice of $\sigma$ follows from the identity $\Omega^k \mu\, \delta^{(k-1)}(\Omega\sigma)  =\mu\,  \delta^{(k-1)}(\sigma)$ which holds for positive functions~$\Omega$.
Finally,  presence of the anomaly term is analyzed by examining for which values of $k$ the quantity $I(\varepsilon)/\varepsilon^k$ includes a $1/\varepsilon$ term when $I(\varepsilon)$ is substituted for its Taylor polynomial.  
\end{proof}

The following proof is inspired by the method outlined in~\cite{BaumJuhl} for the special case of Poincar\'e--Einstein structures.
\begin{proof}[Proof of Theorem~\ref{transform}]
First we note that for 
$\bm \sigma$ and $\bm \tau$ given respectively by $[\mu\, ;\, \sigma]$ and $
[\mu\, ;\, \tau]$, we have $\Vol_{\varepsilon}(D_+;\bm \sigma,\bm\tau)=\int_D \tfrac{\mu}{\sigma^k}  \theta\big( \tfrac\sigma\tau-\varepsilon\big)$ so, following {\it mutatis mutandis} the computation of $I(\varepsilon)$ in the first display of the previous proof, we have
\begin{eqnarray*}
\frac{d \Vol_\varepsilon(D_+; \bm \sigma,e^{-\omega t}\bm \tau)}{dt}&=&
\int_D \frac{\mu}{\sigma^k}\  \omega\,  e^{\omega t}\, \frac{\sigma}{\tau}\, \delta\Big(e^{\omega t}\, \frac\sigma\tau-\varepsilon\Big)\\[1mm]
&=&
\frac{1}{\varepsilon^{k-1}}\, 
\int_D\frac{\mu}{\tau^k}\ \omega e^{k\omega t}\, 
\delta\Big(e^{\omega t}\, \frac\sigma\tau-\varepsilon\Big)
\ =: \frac{ I(\varepsilon;\omega,t)}{\varepsilon^{k-1}}\, .
\end{eqnarray*}
Again, $I(\varepsilon;\omega,t)$ can be expressed in terms of its  Taylor polynomial as
$I(\varepsilon)=I(0)+\varepsilon I'(0) + 
\cdots+
\frac{1}{(k-1)!}\, \varepsilon^{k-1}  I^{(k-1)}(0)+\varepsilon^k R(\varepsilon)$ (suppressing $(\omega,t)$). 
For the  $\varepsilon^{k-1}$ term in this  we find
$$
I^{(k-1)}(0)=(-1)^{k-1}\, \int_D \mu\,  \omega\   \delta^{(k-1)}(\sigma)\, .
$$
In the above we relied on the identity $\delta^{(\ell)}(\alpha x)=\frac1{\alpha^{\ell+1}}\delta^{(\ell)}( x)$.
Focusing on the  $\varepsilon$ independent term in $d \Vol_\varepsilon(D_+; \bm \sigma,e^{-\omega t}\bm \tau)/ dt$ 
gives
$$
\frac{d\Vol_{\rm ren}(e^{-\omega t}\tau)}{dt}=\frac{(-1)^{k-1}}{(k-1)!}\, 
\int_D \mu\,  \omega\   \delta^{(k-1)}(\sigma)\, .
$$
The right hand side is $t$ independent, so
$$
\Vol_{\rm ren}(e^{-\omega}\tau)-\Vol_{\rm ren}(\tau)=
\frac{(-1)^{k-1}}{(k-1)!}\, 
\int_D \mu\,  \omega\   \delta^{(k-1)}(\sigma)\, .
$$
\end{proof}

\addtocontents{toc}{\SkipTocEntry}
\subsection{Example}

Let $a:{\mathbb R}\stackrel{\raisebox{-3mm}{$\sss \rm smooth$}}{
\xrightarrow{\hspace*{1cm}}}{\mathbb R}$, and
consider $M\subset {\mathbb R}^2$ with measure $\mu=  \ext\! x \ext\! y \, \big(1+y a(x)\big)$ (where $M=\{(x,y)|\,  1+y a(x)>0\}$). Let $D$ be the rectangle defined by $-1<x<1$ and $-L<y<L$ where $L$ is chosen such that $D\subset M$. Take as  defining function
$
\sigma = y
$
whence $\Sigma$ is the interval $(-1,1)$ along the $x$-axis. This geometry is depicted below:
\begin{center}
\includegraphics[scale=.37]{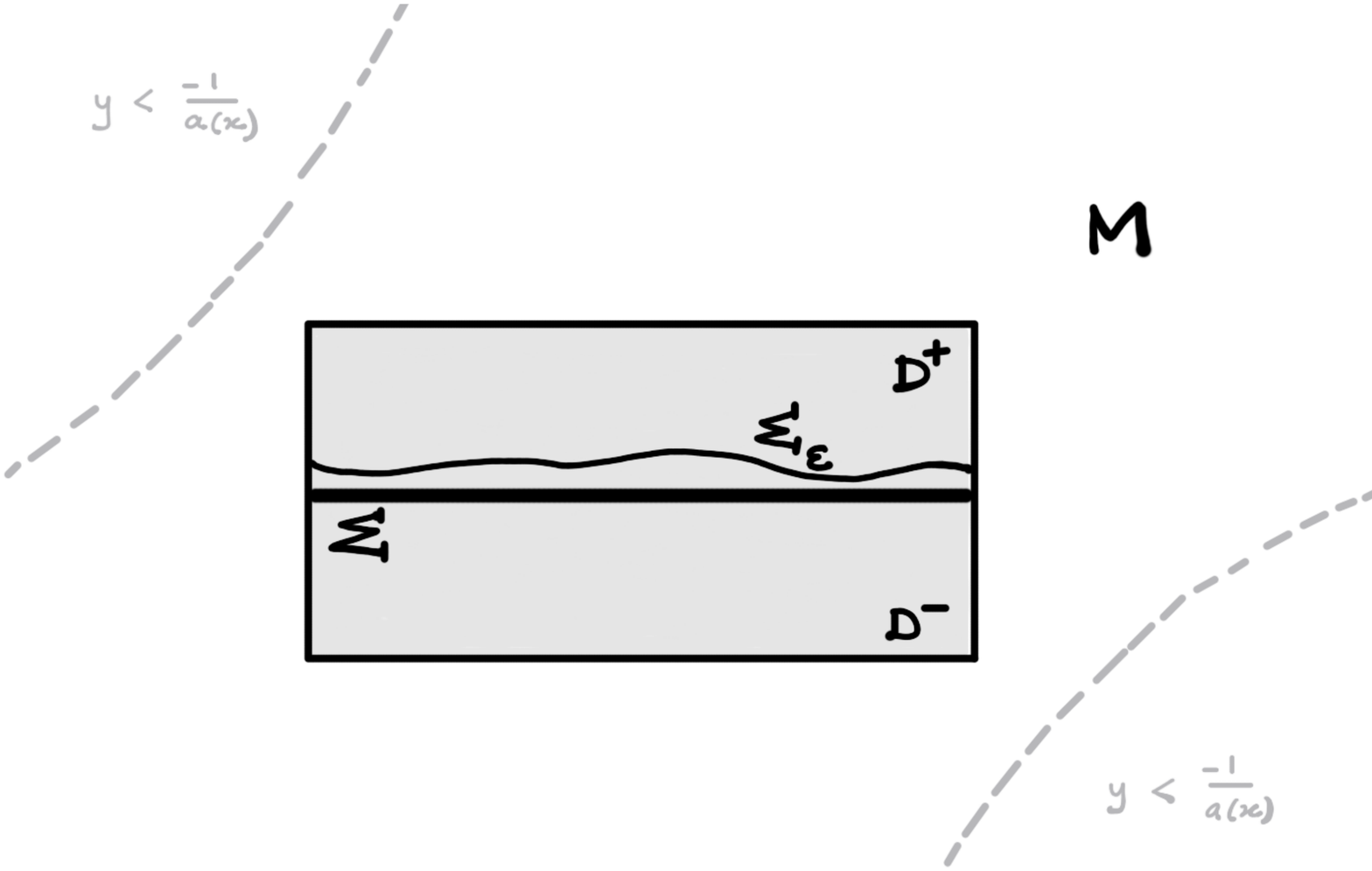}
\end{center}
Let us compute the regulated volume for the case $k=2$ (so $\mu^o=\ext\! x \ext\! \, y (1+a)/y^2$ with respect to the class of regulators $\tau=\tau(x)$. We find (labeling the data $(\bm\sigma,\bm\tau)$ by $(\mu^o,\tau)$)
\begin{eqnarray*}
\Vol_\varepsilon(D_+;\mu^o,\tau)=
\int_{-1}^1\ext\!x \int_{\varepsilon\tau(x)}^L\, \frac{\ext\!y\big(1+y a(x)\big)}{y^2}
=
\frac{\int_{-1}^1 \frac{\ext\!x}{\tau(x)}}\varepsilon
-\log(\varepsilon) \int_{-1}^1 \ext\!x \, a(x)
+\Vol_{\rm ren}
\, ,
\end{eqnarray*}
where the renormalized volume is
$$
\Vol_{\rm ren}=
-\frac2L +\log(L) \int_{-1}^1 \ext\!x\,  a(x)
-\int_{-1}^1 \ext\!x \, a(x) \log \tau(x)
\, .
$$
Observe that the anomaly ${\mathcal A}=-\int_{-1}^1 \ext\!x \, a(x)$,
and that even when this vanishes the regulator dependence of the renormalized volume
 $-\int_{-1}^1 \ext\!x \, a(x)\log\tau(x)$
  is in general not zero. {\it I.e.}, vanishing anomaly need not imply regulator independence of the renormalized volume, but pointwise vanishing of $a(x)$ does.

\section{$Q$ and $T$ curvatures}\label{QTsect}

In the remainder of the paper,  we specialize to
a conformal $d$-manifold $(M,\cc=[g])$, so $k=d$ and~$\bm \mu = \big[{\mu^g } \big]$.
All densities will be conformal densities. 
Our aim in this section is to prove Theorem~\ref{Q+T} but first we need to assemble some basic
differential operators defined on densities as well as key results involving the Laplace--Robin operator and distribution-valued densities. 

For scalar-valued weight~$w=0$ conformal densities~$\bm f:=[g\, ;  f]$, it is convenient to denote  the exterior derivative on functions by
\begin{equation}\label{conformalgradient}
\bm \nabla {\bm f}:=[\hh g\, ; \, \ext\! f]\, .
\end{equation}
We will sometimes call this the {\it conformal gradient operator} and its codomain is covector-valued, weight $w=0$ conformal densities. 
When a weight~$w=1$ conformal density $\bm \tau$ is available, we may form the {\it coupled conformal gradient operator} which is defined acting on scalar-valued densities $\bm f=[g\, ;\, f]$ of arbitrary weight~$w$, according to
\begin{equation}
\label{triangle}
\triangledown^{\scalebox{.7}{$\bm\tau$}} \bm f := [g\, ;\, \tau \nabla f 
-w f \nabla \tau]\, ,
\end{equation}
with  weight $w+1$ covector-valued conformal densities as output. 
If  $\bm \tau$ is a scale,~$\bm \tau^{-1} \triangledown^{\scalebox{.7}{$\bm\tau$}}$ is a {\it Weyl connection}.
When $\bm f$ and $\bm g$ are conformal densities whose product $\bm f \, \bm g$ has weight zero, the conformal and coupled conformal gradients are related by the Leibniz-type rule
$$
\bm \nabla
(\bm f\bm g)=\bm \tau\big(\bm f\, 
\triangledown^{\scalebox{.7}{$\bm\tau$}}\bm g
+
\bm g\, 
\triangledown^{\scalebox{.7}{$\bm\tau$}}\bm f
\big)\, .
$$
If $\bm \tau$ is a true scale with $\bm \tau=
[g\, ;\, 1]$, then the coupled gradient recovers the exterior derivative in the sense $\triangledown^{\scalebox{.7}{$\bm\tau$}} \bm f = [g\, ;\, \ext\! f]$.

\medskip

We also need to handle integrals over a  delta function of a defining density $\bm \sigma$.
If a weight~$1-\dim M$ conformal 
density $\bm f$ on $(M,\cc)$ is an extension of a conformal density~${\bm f}_\Sigma$ of the same weight on the induced hypersurface conformal manifold $({\Sigma},\bm{c}_\Sigma)$ then, up to a measure factor given by the~${\mathcal S}$-curvature, an undifferentiated delta function integration can be performed according to
\begin{equation}\label{deltaintegral}
\int_{\,  D} \bm \delta \, \sqrt{\bm{\mathcal S}} \, \bm f
= \int_\Sigma \bm {f}_\Sigma\, .
\end{equation}
See~\cite[Section 3]{GGHW15} (or~\cite{GWvol}) for the proof of this identity. Also note that this identity already establishes that the first
term in the $\varepsilon$-expansion~\nn{epsilon_expansion}
produces the leading divergence in Equation~\nn{boundaryleading}.

\medskip

Next, we  need to treat differentiated Dirac delta functions. The following proposition (proved in~\cite[Section 3]{GWvol})  allows us to handle up to $(d-2)$ derivatives on the Dirac delta:
\begin{proposition}\label{delta_deriv}
Let~${\mathbb Z}_{>0}\ni j< d-1$ and suppose the~$\bm{\mathcal S}$-curvature is nowhere vanishing. Then 
$$
(\I2^{-1}\!\D)^j \, \bm \delta= (d-j-1)\cdots (d-3)(d-2)\, \bm \delta^{(j)}\, .
$$
\end{proposition}

\medskip

The above proposition implies that 
$$
\D \bm\delta^{(d-2)}=0\, .
$$
Hence, we need a different method to handle the critical case of integrals involving $\bm \delta^{(d-1)}$. The main idea, developed in~\cite{GWvol}, is to introduce a log density into the problem.
The following lemma generalizes Lemma~3.7 of~\cite{GWvol} to include terms that no longer vanish when  $\partial \Sigma\neq \emptyset$.

\begin{lemma}\label{lastderiv}
Let~$\log\bm \tau$ be a weight one log density (where $\bm \tau$ is any true scale) and
suppose the~$\bm{\mathcal S}$-curvature is nowhere vanishing, then
\begin{equation}\label{logform}
\int_{\,  D}\bm{\delta}^{(d-1)}=
-\int_{\, D}
\bm{\delta}^{(d-2)}
{\mathcal L}  \log \bm \tau
 \ + \ 
 \int_{\scalebox{.8}{$\partial D$}} 
 \bm \delta^{(d-2)}\, 
 \frac{\bm{\hat n}_{\scalebox{.8}{$\partial D$}} \, .\, 
 \bm n^{\scalebox{.7}{${\bm\tau}$}}
 }{\bm {\mathcal S}}
 \, ,
\end{equation}
where $\bm n^{\scalebox{.7}{${\bm\tau}$}}=\bm \tau^{-1} 
 \triangledown^{\scalebox{.7}{$\bm \tau$}}
 \bm \sigma$
and the operator
$
{\mathcal L}=\I2^{-1}\circ \D  -
(\bm \nabla_a \bm {\mathcal S}^{-1}) \, {\bm g}^{ab}
  \triangledown^{\scalebox{.7}{$\bm\sigma$}}_b$.
\end{lemma}

\begin{proof}
We only need augment the proof of the corresponding lemma given in~\cite{GWvol} by terms coming from the boundary of $\partial  D$ that were not needed in the closed~$\Sigma$ case studied there. Those only arise in a computation of an integral over the divergence of the conormal $n=\nabla \sigma$ (where $\bm \sigma=[g\, ;\, \sigma]$) which was performed by an integration by parts. We repeat the beginning of that computation here
\begin{eqnarray*}
\int_{ D} {\mu^g }\ {\mathcal S}^{-1} (\nabla.n) \delta^{(d-2)}&=&-\ \int_{\, D} {\mu^g }\ 
\big( {\mathcal S}^{-1} n^2\, \delta^{(d-1)} +\delta^{(d-2)}
 \nabla_n {\mathcal S}^{-1} \big)\\[1mm]
 &&+\ \int_{\partial D} \mu^{g_{\partial  D}}\ 
  \hat n_{\partial D}^a
 \, {\mathcal S}^{-1}n_a\, \delta^{(d-2)}\, .
\end{eqnarray*}
Here $\delta^{(k)}$ is shorthand for $\delta^{(k)}(\sigma)$ and we have denoted $\bm{\mathcal S}=[g\, ;\, {\mathcal S}]$. The last term is the new contribution, and its expression above follows by a straightforward application of the divergence theorem along $\partial D$ with outward unit normal $\hat n_{\partial D}$.
The proof is completed by using the coupled conformal gradient operator to express the $n_a$ in the last term above in terms of conformal densities as  $\bm \tau^{-1} 
 \triangledown_a^{\scalebox{.7}{$\bm \tau$}}
 \bm \sigma$, and then reexamining the proof of Lemma~3.7 of~\cite{GWvol} to determine its coefficient in the contribution to Equation~\nn{logform}.
\end{proof}

\medskip

One of the reasons the  Laplace--Robin operator is so useful is that it enjoys a 
formally self-adjoint property~\cite[Section 2]{GWvol}:

\begin{theorem}\label{parts}
Let~$\bm f$ and~$\bm g$ be densities of weight~$1-d-w$ and~$w$, respectively. Then the Laplace--Robin operator~$\D$ for any weight one density $\bm \sigma=[g\, ;\, \sigma]$ is formally self-adjoint and moreover
$$
\bm f \D \bm g - (\D \bm f)\,  \bm g
+\divergence \bm j=0\, ,
$$
where the weight~$2-d$ covector-valued density 
$$\bm j_a=\big [\hh g\, ;\, \sigma\big(f\, \nabla_a g-(\nabla_a f)\, g\big)-(d+2w-1)\, fg\, \nabla_a \sigma\big]\, .$$
\end{theorem}
\noindent

\medskip

We  also need a variant of Lemma~\ref{lastderiv} for a weight zero density integrated against $\bm \delta^{(d-1)}$
over a closed manifold.
\begin{proposition}\label{trafoform}
Let $(\Lambda,\cc)$ be a closed conformal $d$-manifold, $\bm \sigma$ 
a defining density for a closed,  embedded hypersurface $L$ in $\Lambda$, $\bm f$ a weight zero density and $\bm \tau$ any true scale.
Then 
\begin{equation*}
\int_\Lambda {\bm f}\,  \bm \delta^{(d-1)}= 
-\frac1{(d-2)!}\, 
\int_L \frac1{\sqrt{\bm{\mathcal S}}}\, (\D\I2^{-1})^{d-2}\circ 
\big(\bm f\, 
{\mathcal L}
\,  \log {\bm \tau}
+\bm {\mathcal S}^{-1}
 \nabla_{\bm n^{\scalebox{.6}{$\bm \tau$}}} \bm f
\big)
\, ,
\end{equation*}
where $\mathcal L$ and $\bm n^{\scalebox{.7}{$\bm \tau$}}$ are as given in Lemma~\ref{lastderiv}.
\end{proposition}

\begin{proof}
Again the proof begins with an augmentation of Lemma 3.7 of~\cite{GWvol}, but now to include the presence of $\bm f$. In particular, we need to record instances where derivatives of $\bm f=[g\, ; f]$ (and $\bm \tau=[g\, ;\, \tau]$) are introduced. As before, this is via an integration involving the divergence of the conormal, which is modified as follows:
\begin{multline*}
\int_{ \Lambda} {\mu^g }\, f\,  {\mathcal S}^{-1} (\nabla.n) \delta^{(d-2)}=-\ \int_{\, \Lambda} {\mu^g }\ 
\big(f\,  {\mathcal S}^{-1} n^2\, \delta^{(d-1)} \\+\, \delta^{(d-2)}\, f
 \nabla_n {\mathcal S}^{-1} \ +\ 
\delta^{(d-2)} {\mathcal S}^{-1}\, \nabla_n f 
 \big)\, .
\end{multline*}
There are no surface terms generated by the integration by parts performed above, because $\partial \Lambda=\emptyset$.
In terms of densities, the quantity $\nabla_n f$ can be expressed as $ \bm \tau^{-1}
(\triangledown^{\scalebox{.7}{$\bm \tau$}}
\bm\sigma)
.\bm \nabla f$. Thus we have so far established
$$
\int_{\Lambda}\bm f\, \bm{\delta}^{(d-1)}=
-\int_{\Lambda}\bm{\delta}^{(d-2)}\big[\bm f\, 
\big(\I2^{-1}\circ \D  -
(\bm \nabla_a \bm {\mathcal S}^{-1}) \, {\bm g}^{ab}
  \triangledown^{\scalebox{.7}{$\bm\sigma$}}_b
 \big)  \log \bm \tau
 +(\bm{\mathcal S} \bm \tau)^{-1}
(\triangledown^{\scalebox{.7}{$\bm \tau$}}
\bm\sigma)
.\bm \nabla f
 \big]
 \, .
$$
Now, from Proposition~\ref{delta_deriv} we have
$
\bm \delta^{(d-2)}
\tfrac{1}{(d-2)!}(\I2^{-1}\!\D)^{d-2} \, \bm \delta 
$, so
the proof is completed by
using the formal self-adjointness property of $\D$ (Theorem~\ref{parts}), remembering  that $\partial \Lambda=\emptyset$ and employing Equation~\nn{deltaintegral}.
\end{proof}

\begin{remark}
When $\partial \Sigma=\emptyset$, we can read off the anomaly operator in the conformal  setting directly from the above proposition
$$
\int_\Sigma a(\bm \sigma,\omega)
=
\frac{(-1)^{d-1}}{(d-1)!(d-2)!}\, 
\int_\Sigma \frac1{\sqrt{\bm{\mathcal S}}}\, (\D\I2^{-1})^{d-2}\circ 
\big(\omega\, 
{\mathcal L}
\,  \log {\bm \tau}
+\bm \nabla_{\bm n^{\scalebox{.6}{$\bm \tau$}}} \omega
\big)
\, .
$$
\end{remark}

We are now ready to prove Theorem~\ref{Q+T}.

\medskip

\begin{proof}[Proof of Theorem~\ref{Q+T}]
Our approach is to develop an algorithm to compute the anomaly ${\mathcal A}\propto\int_D\bm \delta^{(d-1)}$.
The first step is to apply Lemma~\ref{lastderiv} which expresses ${\mathcal A}$ as sum of integrals over $D$ and $\partial D$ of the form $\int_D \bm \delta^{(d-2)} {\mathcal L}\log\bm \tau$ and $\int_{\partial D} \bm \delta^{(d-2)} \bm t$ where $\bm t$ is a local, weight zero density. The restriction of the defining density $\bm \sigma$ to $\partial D$ gives, by restriction, a defining density $\bm\sigma|_{\partial D}$ for the hypersurface $\partial \Sigma \hookrightarrow \partial D$. Hence we can apply Proposition~\ref{trafoform} to express $\int_{\partial D} \bm \delta^{(d-2)} \bm t$ as an integral over $\partial \Sigma$ of a local $\bm \tau$-dependent density. 
To perform the integral $\int_D \bm \delta^{(d-2)} {\mathcal L}\log\bm \tau$ we proceed by using  exactly the method as described in the last paragraph of the proof of Proposition~\ref{trafoform}, except now there are boundary terms obtained from applying the divergence theorem to the $\divergence {\bm j}$ term generated when using the formal self-adjoint property for the Laplace--Robin operator. This gives
$$
\int_D\bm \delta^{(d-1)}=
\frac{(-1)^d}{(d-2)!}\,\left[ 
\int_\Sigma \frac1{\sqrt{\bm{\mathcal S}}}\, (-\D\I2^{-1})^{d-2}\circ  
(-{\mathcal L})
\,  \log {\bm \tau}
\, +\, 
\int_{\partial \Sigma} \bm T^{\scalebox{.7}{$\bm\sigma$}}_{\scalebox{.7}{$\bm\tau$}}\ \right]\, .
$$
Locality of $ \bm T^{\scalebox{.7}{$\bm\sigma$}}_{\scalebox{.7}{$\bm\tau$}}$ follows from the algorithm above.
\end{proof}

\begin{proof}[Proof of Corollary~\ref{MIN}]
Given only the data $\cc$ and the embedding $\widetilde \Sigma \hookrightarrow M$, we can 
find~$\bm \sigma=[g\, ;\, \sigma]$ 
such that the corresponding singular metric $g^o=g/\sigma^2$ obeys 
the singular Yamabe condition~\nn{singYam}.
This determines~$\bm \sigma$ up to order~$\bm \sigma^d$~\cite{ACF,CRMouncementCRM,GW15}. Moreover, given a conformally compact manifold with
whose boundary $\widetilde\Sigma$ contains
an embedded hypersurface $\partial \Sigma$, Graham and Witten have proved that  the minimal surface equation for the singular metric $g^o$ formally determines the asymptotics of a defining density $\bm \nu$ 
smoothly up to order~$\bm \sigma^{d-1}$~\cite{GrahamWitten} (see also~\cite{Gra00}). 
This defines a hypersurface $H$ with $\partial H=\partial\Sigma$. Hence we take the region $D_+$ with boundary~$H\cup \Sigma$. In fact, we only need this region in a   solid cylinder neighborhood over $\Sigma$. However, by inspection of the general anomaly Formula~\nn{A} (viewing~$\sigma$ as a coordinate on the solid cylinder), we see that the anomaly is determined by the defining densities $\bm\nu$ and $\bm\sigma$ up to order~$\bm \sigma^{d-1}$.
Thus indeterminancy in the solutions to the singular Yamabe and minimal surface problems do not contribute to the anomaly, so this is canonically determined by the data of the conformal embeddings. 

Naturalness of the  formulae\  for the pair $(\bm Q,\bm T)$ follows easily in view of the integration  by parts algorithm used to establish  Theorem~\ref{Q+T}, and
counting derivatives on the data~$\bm \sigma$ and $\bm \nu$ to  verify that this does not exceed the orders to which these are determined. 

To write natural formulae\  for $(\bm Q,\bm T)$ depending only on the embedding $\partial \Sigma\hookrightarrow \widetilde \Sigma$, one employs the same proof given above, but instead determines the asymptotics of the singular metric~$g^o$ by the Poincar\'e--Einstein condition whereby $g^o$ is an Einstein metric on $M^+$. Here the asymptotics are determined to one order lower~\cite{FGQ,Gra00}, but this still suffices to compute the anomaly.

The Poincar\'e--Einstein holographic formula for 
$Q$-curvature in~\cite{GW} establishes that~$\bm Q$
produces the Branson $Q$ curvature, while 
the analogous extrinsic result was given in~\cite{GWvol}.
\end{proof}

\section{Divergences}\label{Divergences}

The starting point to prove Theorem~\ref{divergences} is 
 the $\varepsilon$-expansion in Equation~\nn{epsilon_expansion}.
To compute the  leading order and nlo 
 divergences, we only need handle at most one derivative of Dirac delta distribution. This is treated in the following Lemma:

\begin{lemma}\label{firstboundarydivergence}
Let ~$d>2$,~${\widetilde\Sigma}$ be a hypersurface with
defining density~${\bm \sigma}$
 and~$ D$ be the region
 described in Theorem~\ref{expansion}. Then, 
 given a true scale~$\bm \tau$, 
$$
\int_{\,  D}
\frac{\bm \delta'}{\bm \tau^{d-2}}=
\frac1{d-2}
\int_{\Sigma}
 \frac1{\sqrt{\bm{\mathcal S}}}\,  \D \Big(\frac{1}{\bm {\mathcal S}  \bm \tau^{d-2} }\Big)
+
\int_{\partial \Sigma} \frac{\cos\theta }{\sqrt{
\bm{  \mathcal S}\hspace{.1mm}
\bm {\mathcal S}_{\scalebox{.7}{$\partial D$}}^{\phantom{.}} \,} \ \bm \tau^{d-2}}\, .
$$
where
$\bm {  \mathcal S}_{\scalebox{.7}{$\partial D$}}$
is the ~$\bm{\mathcal S}$-curvature of
the restriction of~$\bm \sigma$ to~$\partial  D$
and 
$\bm{ \hat n}_{\scalebox{.7}{$\partial D$}}$
is the outward conormal of~$\partial  D$.
\end{lemma}

\begin{proof}
Proposition~\ref{delta_deriv} and the leading term of Theorem~\ref{parts} 
directly imply the first term on the right hand side of the equality displayed above, so we only need focus on the boundary term which, according to Theorem~\ref{parts}, is given by~$-\tfrac1{d-2}\int_{ D} {\mu^g }\, \nabla_a j^a$ 
where (using $\bm \tau=[g\, ;\, 1]$ and $\bm {\mathcal S}=[g\, ;\, {\mathcal S}]$) 
$$j^a=\sigma \big({\mathcal S}^{-1}\nabla^a \delta(\sigma)-(\nabla^a {\mathcal S}^{-1})\delta(\sigma)\big)
-(d-3) n^a {\mathcal S}^{-1} \delta(\sigma)
=-(d-2) n^a {\mathcal S}^{-1} \delta(\sigma)\, .
$$
Thus, via the divergence theorem, the boundary term becomes
$\int_{\scalebox{.7}{$\partial D$}} \ext \! A\,  \hat n_{\scalebox{.7}{$\partial D$}}^a \, n_a {\mathcal S}^{-1} \delta(\sigma)$. 
We may then use Equation~\nn{deltaintegral} to write this as the quoted integral over~$\partial \Sigma$. 
\end{proof}

\begin{proof}[Proof of Theorem~\ref{divergences}]
Apply Equation~\nn{deltaintegral}
to the first term the expansion stated in Equation~\nn{epsilon_expansion} and the above Lemma to the second term thereof. 
\end{proof}

\section{Surface anomaly}
\label{Anomaly}

Before proving Theorem~\ref{anomaly}, we  need a technical Lemma.

\begin{lemma}\label{tl}
Let~$L$ be a closed 1-manifold embedded in a closed conformal 2-manifold~$(N,\bm c)$ with defining density~$\bm \sigma$. Moreover let~$\bm f$ be a weight zero density. Then
$$
\int_N \bm f \hh \bm \delta' = - \int_L {\bm \nabla}_{\!\bm {\hat n}} \big(\bm f/\bm {\mathcal S})\, ,
$$
where $\bm{\hat n}$ is the unit conormal to $L$ determined by $\bm \sigma$ and $\bm {\mathcal S}$ is the~${\mathcal S}$-curvature of~$\bm \sigma$.
\end{lemma}

\begin{proof}
We pick a scale~$g\in \bm c$ then~$\bm \sigma=[\hh g\, ;\, \sigma]$ and note that
$\bm \delta'=\big[\hh g\, ; \big(\nabla_n \delta(\sigma)\big)/n^2\big]$. Integrating by parts (and remembering that~${\mathcal S}=n^2+2\rho\sigma$), we have that
\begin{equation*}
\begin{split}
\int_N \bm f \bm \delta' &=
-\!\int_N \mu^g\,  \delta(\sigma)\, \nabla^a\Big(
\frac{n_a f}{{\mathcal S}-2\rho\sigma}
\Big)
=
-\!\int_N \mu^g \, \delta(\sigma) \Big(\!
 \frac{f\nabla.n +\nabla_n f }{{\mathcal S}}
-\frac{f\nabla_n{\mathcal S}-2\rho n^2 f }{{\mathcal S}^2}
\Big)\\[1mm]
&=-\!\int_N \mu^g\, \delta(\sigma)\, \nabla_n\big(f/{\mathcal S}\big)\, .
\end{split}
\end{equation*}
The last line used that in two dimensions~$-2\rho=\nabla.n+\sigma \J$.
An application of Equations~\nn{conformalgradient} and~\nn{deltaintegral} gives the quoted result.
\end{proof}

\begin{remark}
Specializing to the singular Yamabe setting for which~$\bm \sigma$ is a conformal unit defining density, and setting~$\bm f=1$, the above Lemma then implies  vanishing   anomaly for a line embedded in a 2-manifold.
\end{remark}

\medskip

\begin{proof}[Proof of Theorem~\ref{anomaly}]

By virtue of Equation~\nn{A} we are tasked with computing $\int_{ D}\bm \delta''$. 
We begin by employing Lemma~\ref{lastderiv}. The boundary term given there can be re-expressed as
$$\int_{\scalebox{.8}{$\!\partial D$}} \frac{\bm{\hat n}_{\scalebox{.8}{$\partial D$}} \, .\, (\bm \tau^{-1} 
 \triangledown^{\scalebox{.7}{$\bm \tau$}}
 \bm \sigma)}{\bm {\mathcal S}}\, \bm \delta'
=-\int_{\partial\Sigma}\bm{\hat n}_{\partial\Sigma|\partial  D}^a\bm {\nabla}_a^{\partial D}
\Big(\frac{
\bm{ \hat n}_{\scalebox{.7}{$\partial D$}}  \, .\, 
(\bm \tau^{-1} 
 \triangledown^{\scalebox{.7}{$\bm \tau$}}
 \bm \sigma)
}{\bm {\mathcal S}\hspace{.1mm}\bm {\mathcal S}_{\partial D}}\Big)\, ,
$$
using Lemma~\nn{tl}. Thus we have so far established that
$$
\int_{\,  D}\bm{\delta}''=
-\int_{\, D}
\bm{\delta}'
\big(\I2^{-1}\circ \D  -
(\bm \nabla_a {\mathcal S}^{-1}) \, {\bm g}^{ab}
  \triangledown^{\scalebox{.7}{$\bm\sigma$}}_b
 \big)  \log \bm \tau
-
\int_{\partial\Sigma}\bm{\hat n}_{\partial\Sigma|\partial  D}^a\bm { \nabla}_a^{\partial D}
\Big(\frac{
\bm{ \hat n}_{\scalebox{.7}{$\partial D$}} \, .\, 
(\bm \tau^{-1} 
 \triangledown^{\scalebox{.7}{$\bm \tau$}}
 \bm \sigma)}{\bm {\mathcal S}\hspace{.1mm}\bm {\mathcal S}_{\partial D}}\Big)\, .
$$
It only remains to apply Proposition~\ref{delta_deriv} to the differentiated delta-density~$\bm \delta'$ in the first term on the right hand side and then integrate the resulting Laplace--Robin operator by parts
according to Theorem~\ref{parts} while keeping track of the boundary term, this yields
\begin{multline*}\!\!\!\!\!
\int_{\,  D}\bm{\delta}''
=-\, 
\int_{{\widetilde\Sigma}} \frac1{\sqrt{\bm {\mathcal S}}} \, 
\D\circ\I2^{-1}\circ\big(\I2^{-1}\circ \D  -
(\bm \nabla_a \bm {\mathcal S}^{-1}) \, {\bm g}^{ab}
  \triangledown^{\scalebox{.7}{$\bm\sigma$}}_b
 \big)  \log \bm \tau
 \\[2mm]\qquad\quad\!\!\!\!\!\!\!\!\!\!
-\int_{\partial\Sigma} \Big[\, 
\frac{\bm {\hat n}_\Sigma\, .\, \bm{ \hat n}_{\scalebox{.7}{$\partial D$}} }
{\sqrt{
\scalebox{1}{$ \bm{  \mathcal S}\hspace{.1mm}
\bm {\mathcal S}_{\scalebox{.7}{$\!\partial D$}}^{\phantom{a}}
$}
}}
 \ 
\big(\I2^{-1}\!\circ \D  -
(\bm \nabla_a \bm {\mathcal S}^{-1}) \, {\bm g}^{ab}
  \triangledown^{\scalebox{.7}{$\bm\sigma$}}_b
 \big)  \log \bm \tau
 \ + \ 
 \bm{\hat n}_{\partial\Sigma|\partial  D}^a\bm {\nabla}^{\partial D}_a
\Big(\frac{\bm{\hat n}_{\scalebox{.7}{$\partial D$}} \, .\, 
(\bm \tau^{-1} 
 \triangledown^{\scalebox{.7}{$\bm \tau$}}
 \bm \sigma)}{\bm {\mathcal S}\hspace{.1mm}\bm {\mathcal S}_{\partial D}}\Big)\Big]\, .
\end{multline*}
Remembering that $\bm {\hat n}_\Sigma\, .\, \bm{ \hat n}_{\scalebox{.7}{$\partial D$}} =\cos \theta$ completes the proof.
\end{proof}

\addtocontents{toc}{\SkipTocEntry}
\subsection{Example}\label{EX1}

Consider a surface of revolution~${\widetilde\Sigma}$ in Euclidean 3-space, with metric 
$$
g=\ext\!x^2+\ext\! r^2+r^2\ext\!\theta^2\, .
$$
Moreover, suppose ${\widetilde\Sigma}$ is given by the graph of a function~$f(r)$ so that
$
{\widetilde\Sigma}={\mathcal Z}(\sigma)$, where
$$\sigma=x-f(r)\, .
$$
We require that~$f'(0)=0$ in order that~$\sigma$ is smooth.
For simplicity let us take~${D_+}$ to be the 
cylindrical coordinate region
$$
{D_+}=\{(x,r,\theta)\, |\, x\geq f(r)\, ,\: r\leq R\}
$$ and $ D=\{(x,r,\theta)\, |\,  r\leq R\}$.
We shall try to compute the regulated volume of~${D_+}$ with respect to the defining density
$
\bm \sigma = [\hh g\, ;\, \sigma]
$ and
 the regulator~$\bm \tau=[\hh g\, ;\, 1]$. Note that~$\Sigma=\{(f(r),r,\theta)\, |\, 0\leq r\leq R, 0\leq\theta<2 \pi\}$. This geometry is depicted below:
 
 \begin{center}
 \includegraphics[scale=.26]{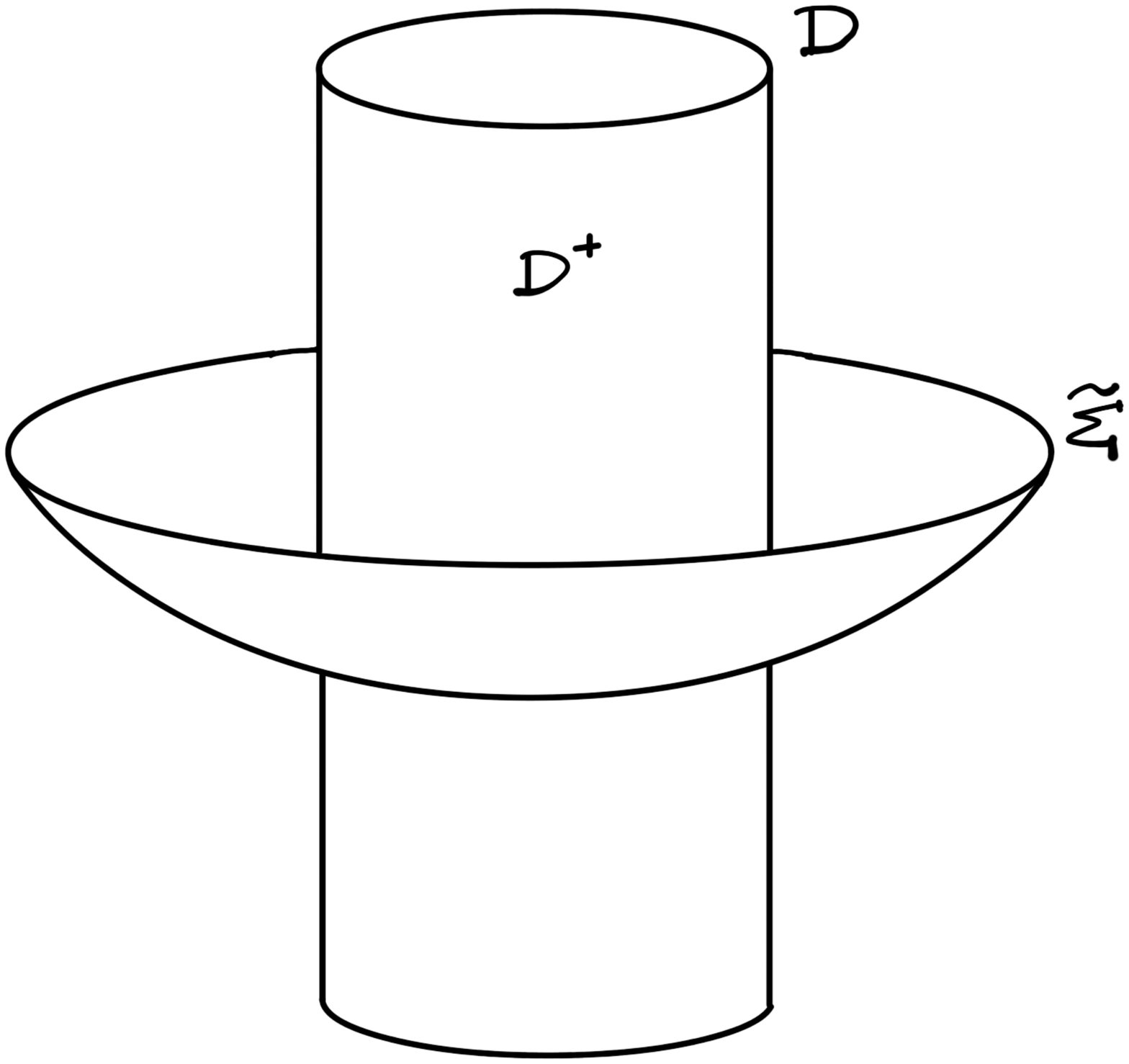}
 \end{center}
 
In fact, in this simple setting we can easily directly evaluate the regulated volume
\begin{equation}\label{VE}
\Vol_\varepsilon=\int_0^{2\pi}
\ext\! \theta
\int_0^R r \ext\! r
\int_{f(r)+\varepsilon}^\infty \frac{\ext\! x}{\big(x-f(\mh r)\big)^3}=\frac{\pi R^2}{2\varepsilon^2}\, .
\end{equation}
Hence both the anomaly and renormalized volume vanish, as does the next to leading (nlo) divergence for this particular regulator.
Indeed, it is not difficult to prove
the same statement for the above problem generalized to surfaces given by the graph of an arbitrary function.
As we shall see, this example provides a nice test
of our result for the leading divergences in~\nn{boundaryleading}
as well as the formula for the surface anomaly~\nn{surfaceanomaly}. 

\medskip

Let us now recalculate the divergences and anomaly. First  note that~$\bm {\mathcal S}=\big[\hh g\, ;\,  1+|\!\ext\! f|^2 +{\mathcal O}(\sigma)\big]$ and  the area element for~${\widetilde\Sigma}$ is ~$\ext \! A=\sqrt{1+|\! \ext\! f|^2}\, r\ext\! r \ext\! \theta$.
Thus~$\int_{\Sigma}
\tfrac1{\scalebox{.7}{$\sqrt{\bm {\mathcal S}}\bm \tau^2$}}
=\int_0^{2\pi}\ext\! \theta \int_0^ R r\ext\! r = \pi  R^2$, so the leading divergence in Equation~\nn{boundaryleading} exactly reproduces~$\Vol_\varepsilon$ in~\nn{VE}. It remains to show that the nlo divergence and anomaly vanish. For the former, according to Equation~\nn{boundaryleading}, we must compute an integral along~$\Sigma$ as well as one along its boundary~$\partial\Sigma$. For the first integrand we compute
$$
\D \Big(\frac1{\bm{\mathcal S}\bm \tau}\Big)
\stackrel{\widetilde\Sigma}=\big[\hh g\hh ;
-(\nabla_n-\rho) {\mathcal S}^{-1}
\big]
\stackrel{\widetilde\Sigma}=\Big[\hh g\hh ;  \frac
{\nabla_n n^2 + 3\rho n^2}
{|n|^4}
\Big]
=\Big[\hh g\hh ; 
\frac{
f''-f'^2f''+\frac{f'}r(1+f'^2)}{(1+f'^2)^2}\Big]\, .
$$
To obtain the displayed result we used that~$\rho=\tfrac13\big(f''+\tfrac 1r f'\big)$,~$n^2=1+f'^2$
and~$\nabla_n n^2=-2f'^2 f''$.
Using the above it follows that
\begin{eqnarray*}
\int_{\Sigma}
 \frac1{\sqrt{\bm{\mathcal S}}}\,  \D \Big(\frac{1}{\bm {\mathcal S}  \bm \tau^{d-2} }\Big)&=&
 \int_0^{2\pi} \!\!\ext\! \theta\int_0^ R r\ext\! r\,  \frac{
f''-f'^2f''+\frac{f'}r(1+f'^2)}{(1+f'^2)^2}\\
&=&
2\pi \int_0^ R \ext\! \Big(
\frac{rf'}{1+f'^2}\Big)=\frac{2\pi R f'( R)}{1+f'( R)^2}\, .
\end{eqnarray*}
This is precisely canceled by the boundary integral because
$$\cos\theta = \bm {\hat n}_\Sigma\, .\, \bm{ \hat n}_{\scalebox{.7}{$\partial D$}} = \left[\hh g\, ;\, -f'/\sqrt{1+f'^2}\right]$$
and
$
\sqrt{
\bm{  \mathcal S}\hspace{.1mm}
\bm {\mathcal S}_{\scalebox{.7}{$\partial D$}}^{\phantom{.}} \,} \ \bm \tau^{d-2}=\left[g\, ;\, 
\sqrt{1+f'^2}
\right]$, while the volume element along~$\partial \Sigma$ is~$ R \ext\! \theta$.

To show that the anomaly vanishes we first compute the bulk contribution to Equation~\nn{surfaceanomaly}
\begin{equation*}
\begin{split}
{\bm Q}^{\scalebox{.7}{$\bm\sigma$}}
&\stackrel{\widetilde\Sigma}=\Big[ g\, ;\, 
-\frac1{|n|^{3}} \big(\nabla_n-3\rho-(\nabla_n\log n^2)\big)\circ (\rho+\nabla_n)\Big(\frac{1}{n^2+2\rho\sigma}\Big)\Big]
\\[1mm]
&\stackrel{\widetilde\Sigma}=\Big[g\, ;\, 
-
\frac{
(1-f'^4)f'(f'''+\tfrac3r f'')+
(1-8f'^2+3f'^4) f''^2
}{(1+f'^2)^{9/2}}
\Big]
\\[1mm]
&=
\Big[g\, ;\,- 
\frac1{r\sqrt{1+f'^2}}\, 
\frac{\ext}{\ext\! r}\Big(
\frac{
f'(f'^3+f'-rf'^2f''+rf'')
}{(1+f'^2)^3}
\Big)
\Big]
\, .
\end{split}
\end{equation*}
Above, in addition to  already gathered data, we used that~$\nabla_n\rho=
-\tfrac13f'(f'''+\tfrac1r f''-\tfrac1{r^2}f')$
and~$\nabla_n^2 n^2 =2f'^2(f'f'''+2f''^2)$. 
Hence the bulk contribution to the anomaly is
$$
\int_{\Sigma}
\, 
 {\bm Q}^{\scalebox{.7}{$\bm\sigma$}}
= -
\frac{\pi
f'( \mh R)\big(f'(\mh R)^3+f'(\mh R)-Rf'(\mh R)^2f''(\mh R)+Rf''(\mh R)
\big)}{(1+f'(\mh R)^2)^3}\, .
$$
To compute the boundary contribution to the anomaly we first note that
$$
{\mathcal L}\,   \log \bm \tau
 \stackrel{\partial\Sigma}=
\Big[g\, ;\,
-\frac{f'(\mh R)^3+f'(\mh R)-5Rf'(\mh R)^2f''(\mh R)+Rf''(\mh R)}{3R(1+f'(\mh R)^2)^2} 
\Big]\, . 
$$
Also~$\bm{\hat n}_{\scalebox{.8}{$\partial D$}} \, .\, \bm n^{\scalebox{.7}{${\bm\tau}$}}
 =[g\, ;\, -f']$ 
so that
\begin{multline*}
 \bm{\hat n}_{\partial\Sigma|\partial  D}^a\bm { \nabla}^{\partial D}_a
\Big(\frac{\bm{\hat n}_{\scalebox{.8}{$\partial D$}} \, .\, \bm n^{\scalebox{.7}{${\bm\tau}$}}
}{\bm {\mathcal S}\hspace{.1mm}\bm {\mathcal S}_{\partial D}}\Big)
\\
 =\Big[g\, ;\, 
\frac{\partial}{\partial x}\Big(\frac{-f'}{1+f'^2+\tfrac23(x-f)(f''+\tfrac1r f')}\Big)\Big]
\stackrel{\partial\Sigma}=\Big[g\, ;\, 
\frac{2f'(\mh R)(Rf''(\mh R)+f'(\mh R))}{3R(1+f'(\mh R)^2)^2}\Big]\, .
\end{multline*}
Noting that
$\cos\theta
\big/\!
\scalebox{.8}{$
\sqrt{
\bm{  \mathcal S}\hspace{.1mm}
\bm {\mathcal S}_{\scalebox{.7}{$\partial D$}}^{\phantom{.}} \,} 
$}
=\left[\hh g\, ;\, -f'/(1+f'^2)\right]$ we find
for the boundary  contribution
to the anomaly~$$
\int_{\partial\Sigma} \, 
 {\bm T}^{\scalebox{.7}{$\bm\sigma$}}
= 
\frac{\pi
f'( \mh R)\big(f'(\mh R)^3+f'(\mh R)-Rf'(\mh R)^2f''(\mh R)+Rf''(\mh R)
\big)}{(1+f'(\mh R)^2)^3}\, .
$$
This exactly cancels the result for
 $\int_{\Sigma}
\, 
 {\bm Q}^{\scalebox{.7}{$\bm\sigma$}}$ found above.
Hence, we find that the anomaly~${\mathcal A}=0$
in agreement with our exact computation of the regulated volume.

\section{Singular Yamabe surface anomaly}
\label{Yamabe}

We now consider the case where the singular metric $g^o$  obeys the singular
 Yamabe condition~\nn{singYam}.
 In fact, via the uniqueness results of~\cite{ACF,GW15}, this implies that the anomaly is completely determined by the conformal embedding of ${\widetilde\Sigma}\hookrightarrow M$. Remembering that the singular Yamabe condition can be 
 reformulated as a unit condition for the ${\mathcal S}
$-curvature as in Equation~\nn{one}. Thus, the result of Theorem~\ref{anomaly} for the surface anomaly simplifies to 
$$
{\mathcal A}=\frac12\int_{\Sigma}\D^2 \log \bm \tau 
+ \frac 12 \int_{\partial\Sigma}
\, {\bm T}^{\scalebox{.7}{$\bm\sigma$}}\, ,
$$
where
\begin{equation}\label{J}
{\bm T}^{\scalebox{.7}{$\bm\sigma$}}=
 \left(
\frac{\cos\theta }{\sqrt{
\scalebox{.8}{$\bm {\mathcal S}_{\scalebox{.7}{$\partial D$}}^{\phantom{.^.}}~$}} } \ 
 \D   \log \bm \tau
\, + \, 
 \bm{\hat n}_{\partial\Sigma|\partial  D}^a\bm { \nabla}^{\partial D}_a
\Big(\frac{\bm{\hat n}_{\scalebox{.8}{$\partial D$}} \, .\, 
 \bm n^{\scalebox{.7}{${\bm\tau}$}}
}{\bm {\mathcal S}_{\partial D}}\Big)
\right)\, .
\end{equation}
The integrand of the bulk $\int_{\Sigma}$ term is the extrinsic~$Q$-curvature of~\cite{GW15}, and has been computed in~\cite{GW161}. This yields for the bulk contribution 
$
\tfrac12\int_{\Sigma}\big(\,\,  \bar{\!\!\J}\, -\tfrac12\,  \bm\IIo_{ab}^\Sigma\,  \bm \IIo_\Sigma^{ab} \big)
$, where $\, \bar{\!\!\J}=\Sc^{\bar g}/2$. 
Hence, 
$$
\bm Q^{\scalebox{.7}{$\bm\sigma$}}_{\scalebox{.7}{$\bm\tau\!=\![g ; 1]$}}=\bar{\!\!\J}\, -\tfrac12\,  \bm\IIo_{ab}^\Sigma\,  \bm \IIo_\Sigma^{ab}\, .
$$
Thus it only remains to compute the weight~$w=-1$ density~${\bm T}^{\scalebox{.7}{$\bm\sigma$}}$ above to complete the proof of Theorem~\ref{singanomaly}.

\medskip

Before proving Theorem~\ref{singanomaly}, however,  we need to gather together 
some data from the geometry of intersecting surfaces.
Let us consider a pair of smooth surfaces~${\widetilde\Sigma}$ and~$\Lambda$  embedded in a Riemannian three-manifold~$(M,g)$ as depicted below:

\begin{center}
\includegraphics[scale=.3]{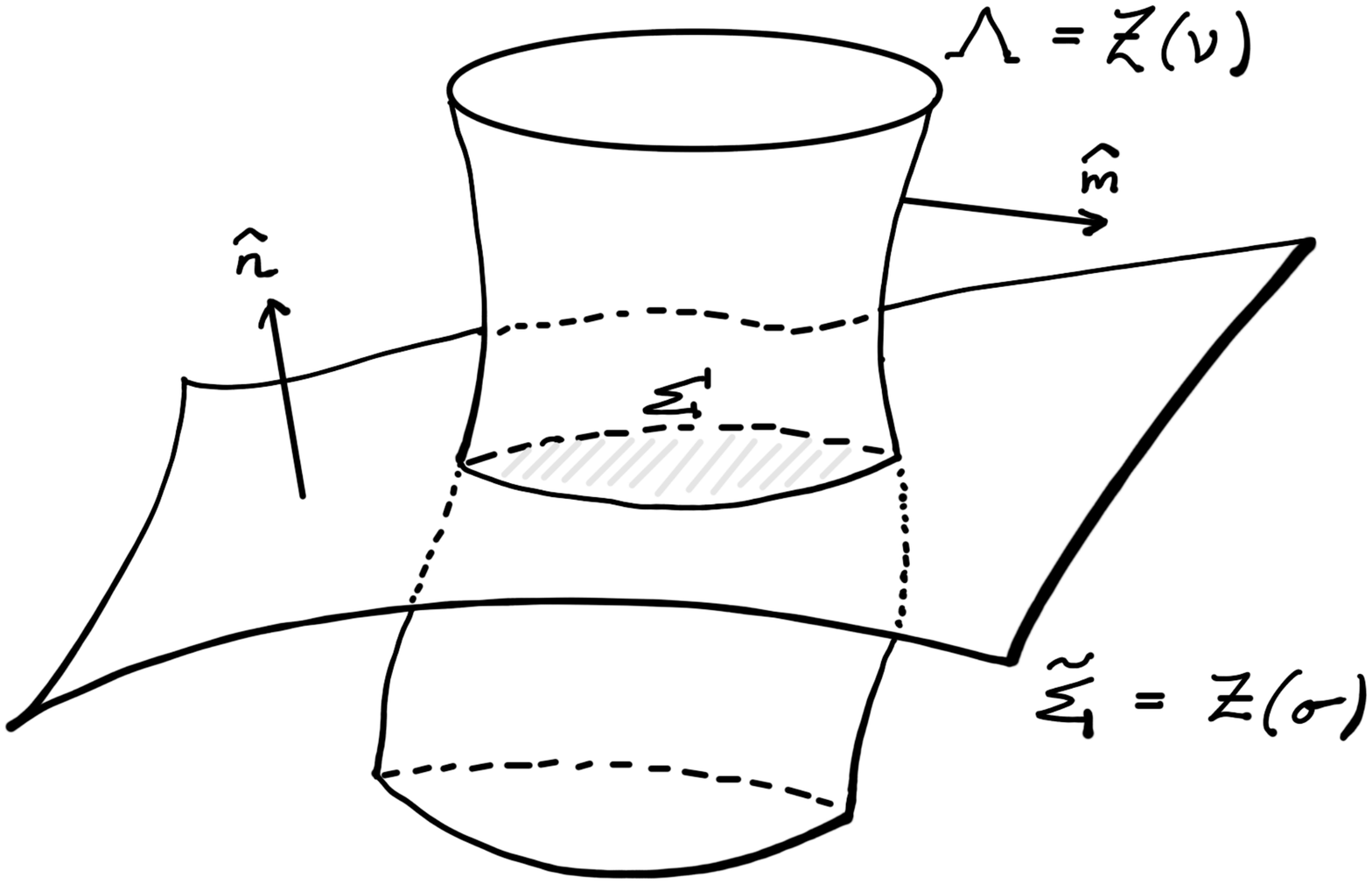}
\end{center}

Here we have depicted a solid cylinder topology of the type assumed by Theorem~\ref{expansion} to hold
in a neighborhood of~$\Sigma$. 
In the context of Theorem~\ref{singanomaly}, the surface~$\Lambda$ plays the role of $\partial  D$ in that neighborhood.
Also let us assume that their intersection ${\widetilde\Sigma}\cap \Lambda$ is the boundary of some  region $\Sigma \subset {\widetilde\Sigma}$.
In what follows, we assume that these surfaces intersect at a non-zero angle.

Moreover it will be useful to suppose~${\widetilde\Sigma}={\mathcal Z}(\sigma)$ and~$\Lambda={\mathcal Z}(\nu)$ where~$\sigma$ and~$\nu$ are respective defining functions. Also,  the following covectors
\begin{equation}\label{pres}
\hat n:=\frac{\nabla \sigma}{|\nabla\sigma|}\, ,\qquad
\hat m:=\frac{\nabla \nu}{|\nabla\nu|}\, ,
\end{equation}
are preinvariants (in the sense of~\cite{GW15}) for the respective unit conormals. ({\it I.e.}, 
$\hat n|_{\widetilde \Sigma}$ 
and 
$\hat m|_\Lambda$ 
are conormals to ${\widetilde\Sigma}$ and $\Lambda$, respectively, and this holds for {\it any} choice of defining function $\sigma$ or $\lambda$ for each of the two surfaces).
We  will use the notation~$n=\nabla \sigma$ and~$m=\nabla\nu$
and as usual adorn various quantities with a~${\widetilde\Sigma}$ or~$\Lambda$ to indicate which hypersurface they correspond to.
In addition, since we are interested in the singular Yamabe setting, we assume that~$[g\, ;\, \sigma]$ is a conformal unit defining density so that
$$
|n|^2=1-2\rho \sigma +{\mathcal O}(\sigma^3)\, ,
$$ 
where~$\rho:=-\tfrac 13 (\Delta \sigma + \J\,  \sigma)$ and~$-\rho|_\Sigma=H_\Sigma$ where~$H_\bullet$ denotes mean curvature, and $\J=\tfrac14\Sc^g$.
Without loss of generality (again see~\cite{GW15}), we may take~$\nu$ to be a {\it unit defining function} meaning that $|\nabla \nu|=1$.

The Euler characteristic~$\chi_{\Sigma}$ of the surface~$\Sigma$ is given, according to the Chern--Gau\ss--Bonnet theorem, by
\begin{equation}\label{GB}
2\pi \chi_{\Sigma} = \int_{\Sigma} \J_\Sigma 
+ \int_{\partial\Sigma} \kappa\, ,
\end{equation}
where~$\kappa$ is the geodesic curvature of~$\partial \Sigma$, or in other words, the mean curvature of the curve~$\partial \Sigma$ viewed as an embedding in~${\widetilde\Sigma}$ (with metric induced by~$g$). Thus
\begin{equation}\label{kappa}
\kappa=\nabla_\Sigma^a \, \hat p_a
|_{\partial\Sigma}\, ,
\end{equation}
where 
$$\hat p=\frac{\hat m -  \cos\theta\, \hat n}{\sin\theta}$$ is the unit normal to~$\partial \Sigma\hookrightarrow {\widetilde\Sigma}$ and
 the conformal hypersurface invariant~$\theta$ of~$\partial \Sigma$ is defined by the preinvariant 
$$
\hat n.
\hat m :=\cos\theta\, .
$$
It will also be useful to introduce the unit normal of~$\partial \Sigma\hookrightarrow \Lambda$, this is given by
$$
\hat q = \frac{\hat n - \cos\theta\, \hat m}{\sin\theta}\, .
$$
The preinvariants of Equation~\nn{pres} will be used to define~$\hat n$,~$\hat m$,~$\hat q$,~$\hat p$ and~$\theta$ in the bulk, and we will often use this notation without comment below.

We first need the following technical results:
\begin{lemma}\label{first}
Let~$\cos\theta=(\nabla \sigma).(\nabla \nu)/(|\nabla\sigma| |\nabla \nu|)$. Then 
$$
\nabla_{\hat m} \cos \theta\stackrel{\widetilde\Sigma}=\IIo^\Sigma_{\hat m\hat m}
+\sin^2\theta\, H_\Sigma
\, ,\qquad
\nabla_{\hat n} \cos \theta\stackrel\Lambda=\IIo^\Lambda_{\hat n\hat n}+\sin^2\theta\, H_\Lambda\, .
$$
\end{lemma}
\begin{proof}
For the first identity, we compute
\begin{eqnarray*}
\nabla_{\hat m}(\hat n.\hat m)&\stackrel{\widetilde\Sigma}=&
\hat m^a (\II_{ab}^\Sigma+\hat n_a\nabla_{\hat n} \hat n_b)\hat m^b
+\hat m^a  \hat n^b (\II_{ab}^\Lambda+\hat m_a\nabla_{\hat m} \hat m_b)
\\[1mm]
&\stackrel{\widetilde\Sigma}=&
\IIo_{\hat m\hat m}+\hat m^a \hat m^b H_\Sigma (g_{ab}-\hat n_a \hat n_b)\, .
\end{eqnarray*}
This yields the first result quoted 
and the second follows by symmetry.
\end{proof}
\noindent
Two further identities are obvious corollaries of this lemma:
\begin{eqnarray}\label{further}
\nabla_{\hat p}\cos\theta&\stackrel{\partial\Sigma} =&
\sin\theta\, \big(\, 
\IIo^\Sigma_{\hat p \hat p}\:  +\, H_\Sigma
-\cos\theta\, (\, \IIo^\Lambda_{\hat q\hat q}+ H_\Lambda)\big)
\, ,\\[2mm]
\nabla_{\hat q}\cos\theta&\stackrel{\partial\Sigma} =&
\sin\theta\, \big(\, 
\IIo^\Lambda_{\hat q \hat q}\   +H_\Lambda
-\cos\theta\, (\, \IIo^\Sigma_{\hat p\hat p}+ H_\Sigma)\big)
\, .\nonumber
\end{eqnarray}
The second of the above identities can also be obtained as a consequence of the following more general result.
\begin{lemma}\label{A3}
$$
\nabla_{\hat q} \hat m_a\stackrel\Lambda=
\IIo^\Lambda_{a\hat q}
+H_\Lambda \hat q
\, ,\quad
\nabla_{\hat q}\hat n_a\stackrel{\widetilde\Sigma}=
 -\cos\theta\,  (\, \IIo^\Sigma_{a\hat p}
+ H_\Sigma \, \hat p_a)\, .
$$
\end{lemma}
\begin{proof}
The first identity is straightforward since~$\nabla_{\hat q}=\tfrac{1}{\sin\theta}\,  \hat n^a\, (\nabla
_a-\hat m_a \nabla_{\hat m})$ implies that
$$
\nabla_{\hat q} \hat m_a\stackrel\Lambda=
\frac{\IIo^\Lambda_{a\hat n}
+H_\Lambda(\hat n_a - \cos\theta\, \hat m_a)}{\sin\theta}\, .
$$
For the second some computation is needed:
\begin{eqnarray*}
\nabla_{\hat q}\hat n_a&\stackrel{\widetilde\Sigma}=&
\frac{\hat q^b (\II_{ab}^\Sigma + \hat n_b \nabla_{\hat n} \hat n_a)}{\sin\theta}
\ \stackrel{\widetilde\Sigma}= \ -\cot\theta\, \hat m^b \big(\IIo^\Sigma_{ab}+H_\Sigma(g_{ab}-\hat n_a\hat n_b)
\big)\\[1mm]
&\stackrel{\widetilde\Sigma}=&-\cot\theta\,  \big(\, \IIo^\Sigma_{a\hat m}
+ H_\Sigma (\hat m_a
-\cos\theta\,  \hat n_a)\big)\, .
\end{eqnarray*}
The quoted result now follows easily from the last expression above. 
\end{proof}

\medskip

Now we can give a useful formula for the geodesic curvature:
\begin{proposition}\label{geodesic}
The geodesic curvature of~$\partial \Sigma\hookrightarrow {\widetilde\Sigma}$ is  given by
$$
\kappa_{\partial\Sigma}\  \stackrel{\partial\Sigma}=\ 
\frac{H_\Lambda-\IIo^\Lambda_{\hat q\hat q}
-\cos\theta\,  (H_\Sigma-\IIo^
{\Sigma}_{\hat p\hat p})}{\sin\theta}
\, .
$$
\end{proposition}

\begin{proof}
We begin by massaging  Expression~\nn{kappa}:
\begin{eqnarray*}
\kappa_{\partial\Sigma}
&\stackrel{\partial\Sigma}=&
(\nabla^a - \hat n^a \nabla_{\hat n} )\Big(
\frac{\hat m_a -   \cos\theta\, \hat n_a}{\sin\theta}\Big)\\
&\stackrel{\partial\Sigma}=&
\frac{
2H_\Lambda
-\II^\Lambda_{\hat n\hat n}-2 \cos\theta\,  H_\Sigma
}{\sin\theta}
- \nabla_{\hat p} \log\sin\theta\\[2mm]
&\stackrel{\partial\Sigma}=&
\frac{
2H_\Lambda
-2 \cos\theta\,  H_\Sigma
}{\sin\theta}
-\sin\theta\, (\IIo^\Lambda_{\hat q\hat q}+H_\Lambda)
- \nabla_{\hat p} \log\sin\theta
\, .
\end{eqnarray*}
Using the second of the two identities in Equation~\nn{further} we have
$$
\nabla_{\hat p} \log\sin\theta\, 
\stackrel{\partial\Sigma}=\, 
-\cot\theta
\big(\, 
\IIo^\Sigma_{\hat p \hat p}\:  +\, H_\Sigma
-\cos\theta\, (\, \IIo^\Lambda_{\hat q\hat q}+ H_\Lambda)\big)
\, .
$$
Combining the above ingredients gives the result.
\end{proof}

We now proceed to our computation of the boundary contribution to the anomaly, and begin with a pair of technical lemmas:
\begin{lemma}\label{SL}
The~${\mathcal S}$-curvature of the pullback of~$\sigma$ and~$g$ to~$\Lambda$ evaluated along~$\partial \Sigma$ is given by 
$$
{\mathcal S}_{\Lambda}\stackrel{\partial\Sigma}=
\sin^2\theta\, .
$$
\end{lemma}

\begin{proof}
For later use, we compute the~${\mathcal S}$-curvature of~$\Lambda$\
to accuracy~${\mathcal O}(\sigma^2)$ along~$\Lambda$:
\begin{eqnarray}
{\mathcal S}_{\Lambda}
&=&
|\nabla \sigma - \hat m \nabla_{\hat m}\sigma|^2
-\sigma\big(
(\nabla^a - \hat m^a \nabla_{\hat m})
(\nabla_a - \hat m_a \nabla_{\hat m})\sigma
+\J_\Lambda \, \sigma
\big)
\nonumber
 \\[1mm]
&=&
|n|^2(1-(\hat n.\hat m)^2)-
\sigma\big(\Delta \sigma
-\hat m^a \nabla_{\hat m}n_a-\hat m.n\,  (\nabla_a \hat m^a)
\big)
+{\mathcal O}(\sigma^2)\, .\label{accurate}
\end{eqnarray}
Since~$\sigma=0$ and~$|n|=1$ along~$\partial\Sigma$, and~$\hat n.\hat m=\cos\theta$, the quoted result now follows.
\end{proof}

\begin{lemma}
$$
\nabla_{\hat q}\, {\mathcal S}_{\Lambda}
\stackrel{\partial\Sigma}=
\sin\theta\big((1+\cos^2\theta)\, \IIo_{\hat p\hat p}^\Sigma
-2\cos\theta\, 
\IIo^\Lambda_{\hat q\hat q}\big)\, .
$$
\end{lemma}
\begin{proof}
Our starting point is Equation~\nn{accurate}. From there and using
~$$\nabla_{\hat q}\sigma\stackrel{\partial\Sigma}=\frac{1
-(\hat n.\hat m)^2}{\sin\theta}=\sin\theta
\, ,$$ we have
\begin{eqnarray*}
\nabla_{\hat q}\, {\mathcal S}_{\Lambda}
&\stackrel{\partial\Sigma}=&
2 \sin^3\theta\, H_\Sigma
-2\cos\theta\, \nabla_{\hat q} \cos\theta
\\[1mm]&&
-\sin\theta
\big(
3H_\Sigma 
-\hat m^a \hat m^b (\IIo_{ab}^\Sigma+ g_{ab}H_\Sigma)
-2\cos\theta\,  H_\Lambda
\big)
\, .
\end{eqnarray*}
Employing Equation~\nn{further},
the result follows easily.
\end{proof}

We now have enough ingredients to compute the boundary integrand~${\bm T}^{\scalebox{.7}{$\bm\sigma$}}$ of Equation~\nn{J}.
\begin{proposition}
The weight~$w=-1$ density ${\bm T}^{\scalebox{.7}{$\bm\sigma$}}$ is given by
$$
{\bm T}^{\scalebox{.7}{$\bm\sigma$}}
_{\scalebox{.7}{$\bm\tau\!=\![g ; 1]$}}=
\kappa_{\partial \Sigma}
+
\frac{\,2( \bm \IIo^\Lambda_{\scalebox{.7}{$\bm{\hat q}\bm {\hat q}$}}-\cos\theta\ \bm \IIo^\Sigma_{\scalebox{.7}{$\bm{\hat p}\bm {\hat p}$}} )}{\sin^3\theta}
-\cot\theta\  {\bm \IIo^\Sigma_{\scalebox{.7}{$\bm{\hat p}\bm {\hat p}$}}}\, .
$$
\end{proposition}

\begin{proof}
Let us denote $T^\sigma$ according to
${\bm T}^{\scalebox{.7}{$\bm\sigma$}}=:[g\, ; T^\sigma]$.
Then 
since~$[g\, ;\, \sigma]$ is a conformal unit defining density it follows from Equation~\nn{logact} that 
$\D  \log \bm \tau\stackrel{{\widetilde\Sigma}} = [g\, ;\, -H_\Sigma ]$
for~$\bm \tau=[g\, ;\, 1]$. Thus we have
\begin{eqnarray*}
T^\sigma&\stackrel{\partial\Sigma}=&-\cot\theta\, H_\Sigma+\nabla_{\hat q}\big(\hat m. n/{\mathcal S}_\Lambda\big)\\[1mm]
&\stackrel{\partial\Sigma}=&
\frac{
\,
\IIo^\Lambda_{\hat q \hat q}\   +H_\Lambda
-\cos\theta\, (\, \IIo^\Sigma_{\hat p\hat p}+ H_\Sigma)
}{\sin\theta}
-\frac{\cos\theta\, \big((1+\cos^2\theta)\, \IIo_{\hat p\hat p}^\Sigma
-2\cos\theta\, 
\IIo^\Lambda_{\hat q\hat q}\big)
}
{\sin^3\theta}
\end{eqnarray*}
Comparing the above display to the result for the geodesic curvature in Proposition~\ref{geodesic}
it follows that
\begin{eqnarray*}
T^\sigma
&\stackrel{\partial\Sigma}=&
\kappa\, +\, 
\frac{2\, \IIo^\Lambda_{\hat q\hat q}-\cos\theta\, (2+\sin^2\theta)\, \IIo^\Sigma_{\hat p\hat p}}{\sin^3\theta}
\, .
\end{eqnarray*}
Expressing the above display in terms of densities gives the quoted result.
\end{proof}

\begin{proof}[Proof of Theorem~\ref{singanomaly}]
Theorem~\ref{singanomaly}.
is proved by 
combining the above Proposition
with Equation~\nn{J} and the Chern--Gau\ss--Bonnet Formula~\nn{GB}.
\end{proof}

\begin{remark}
The divergences given by Theorem~\ref{divergences} (specialized to $d=3$) are simple to compute in the regulator~$\bm \tau=[g\, ;\, 1]$ by using Equation~\nn{LR} for the Laplace-Robin operator and
Lemma~\ref{SL} to handle the ${\mathcal S}$-curvature of~$\partial D$. The result is
\begin{equation}\label{div3}
\Vol_\varepsilon=\frac{A_{\Sigma}}{2\varepsilon^2}
+\frac1\varepsilon
\, \Big(\int_{\Sigma} H_\Sigma
-\int_{\partial\Sigma}
\cot\theta
\Big)+\cdots \, .
\end{equation}
Here $A_{\Sigma}$ is the area of the hypersurface $\Sigma$ measured by the metric choice $g$.
\end{remark}

\addtocontents{toc}{\SkipTocEntry}
\subsection{Example}

We simplify the example of Section~\ref{EX1}
by specializing the surface of revolution to a paraboloid with defining function
$$
\sigma=x-\frac12 r^2\, .
$$
This amounts to the choice $f(r)=\tfrac12 r^2$ in the previous example, and gives a geometry simple enough that we can
easily solve its singular Yamabe problem. We keep the regions~${D_+}$ and~$ D$ the same as the earlier example.

Since the anomaly only depends on the conformal embedding, we can compute it directly from Theorem~\ref{singanomaly}. We have listed the relevant geometric data  below
\begin{eqnarray*}
\chi_{\partial\Sigma} =1\, , \quad \sqrt{\det g_\Sigma} =r\sqrt{1+r^2}\, ,\!\!\!&&\!\!\!
\cos\theta =  - \frac{R}{\sqrt{1+R^2}}\, ,
\quad
\sin\theta=\frac{1}{\sqrt{1+R^2}}\, ,\\[3mm]
\bm\IIo_{ab}^\Sigma\,  \bm \IIo_\Sigma^{ab}=
[g\, ;\, \tfrac12 r^4/(1+r^2)^3] & \Rightarrow &
\int_{\Sigma} \bm\IIo_{ab}^\Sigma\,  \bm \IIo_\Sigma^{ab}=-\frac{\pi}{3}\Big(8-\frac{3R^4+12R^2+8}{\sqrt{(1+R^2)^3}}\Big)\, ,\\[3mm]
\bm \IIo^\Sigma_{\scalebox{.7}{$\bm{\hat p}\bm {\hat p}$}}=\frac12\frac{R^2}{\sqrt{(1+R^2)^3}}\, ,\!\!\!&&\!\!\!
\bm \IIo^\Lambda_{\scalebox{.7}{$\bm{\hat q}\bm {\hat q}$}}
=-\frac{1}{2R}\, .
\end{eqnarray*}
Applying Theorem~\ref{singanomaly} we then find
the anomaly 
\begin{equation}
\label{expA}
{\mathcal A}=\frac{5\pi}3\Big(1-
\frac{1+\tfrac{12}5 R^2+\tfrac{21}{20}R^4}{\sqrt{(1+R^2)^3}}
\Big)\, .
\end{equation}
To compute the divergences 
we need the area and mean curvature of $\Sigma$ in the metric $g$ and its integral over $\Sigma$. These are easy to compute and are given by 
$$
H_\Sigma=-\frac12\frac{r^2+2}{\sqrt{(1+r^2)}}\:\: \mbox{ and } \:\:
\int_{\Sigma} H_\Sigma=
-\frac{\pi}2 \, \big(R^2+\log(1+R^2)\big)\, . 
$$
Also the area of $\Sigma$ in the metric $g$ is given by
$$
A_{\Sigma}=\frac{2\pi}{3}\big(\sqrt{(1+R^2)^3}-1\big)\, .
$$
Using the above data, Equation~\nn{div3}
gives  the divergences in the regulated volume:
\begin{equation}\label{expdiv}
\Vol_\varepsilon=\frac{\pi\, \big(\sqrt{(1+R^2)^3}-1\big)}{\varepsilon^2}
+\frac{\pi\, \big(3 R^2-\log(1+R^2)\big)}{2\varepsilon}+\cdots\, .
\end{equation}

\medskip

The regulated volume integral  in this example is again simple enough that we can compute its divergences and anomaly explicitly. 
We begin by solving the singular Yamabe condition~\nn{singYam}.
For that we
employ the recursion of~\cite{GW15} (see the Appendix of~\cite{GW161,GGHW15} for worked examples) and find that the density~$\bm {\bar \sigma}=[\hh g\, ;\bar \sigma]$ with
\begin{equation}\label{barsi}
\bar \sigma=\hat\sigma
\left(
1-\frac{\hat\sigma}4\, \frac{3r^2+2}{\sqrt{(1+r^2)^3}}
+\frac{\hat\sigma^2}{6}\, \frac{r^2(5r^2+6)}{(1+r^2)^3}
\right)\, ,
\mbox{ where }
\hat\sigma :=\frac{\sigma}{|\nabla\sigma|}=\frac{x-\tfrac12 r^2}{\sqrt{1+r^2}}\, ,
\end{equation}
obeys the singular Yamabe condition
$$\bm {\mathcal S}_{\scalebox{.7}{$\bar{\bm \sigma}$}}=1+\bar{\bm \sigma}^3 \bm B+{\mathcal O}(\bar{\bm \sigma}^4)\, ,$$
with~$${\bm B}= \Big[g\, ;\, 
\frac{1}{12}\frac{r^6+6r^4+24r^2-16}{\sqrt{(1+r^2)^3}}
\Big]\, .$$
Note that this gives an explicit formula for the obstruction density studied in~\cite{GW15,Grahamnew,GWvol,GW161}.
Let us take~${D_+}$ again to be the cylindrical coordinate region of the previous example so that~$\bm{ \hat n}_{\scalebox{.7}{$\partial D$}} =[\hh g\, ;\,  \ext\! r]$.

In the regulator determined by the~$\bm \tau=[\hh g\, ; \, 1]$, the regulating hypersurface~$\Sigma_\varepsilon$  is the zero locus of~$\bar \sigma-\varepsilon$. We can solve this as an equation for~$x$ by quadratures and find that
$$
\bar x=\frac 12 r^2 + \sqrt{1+r^2} \  \varepsilon + \frac14 \frac{3r^2+2}{1+r^2}\  \varepsilon^2
+\frac{1}{24}\, \frac{7r^4+12r^2+12}{\sqrt{(1+r^2)^5}}\, \varepsilon^3
=:\frac12 r^2 +\bar \varepsilon(r)
$$
obeys~$\bar \sigma(\bar x)=\varepsilon + {\mathcal O}(\varepsilon^4)$.
So now, using~Equation~\nn{barsi}, we have
$$
\Vol_\varepsilon=
2\pi \int_0^R r\ext\! r  \int_{\bar x(r)}^\infty \, \frac{\ext\!x}{\bar\sigma^3}\ =\ 
2\pi \int_0^R r\ext\! r \sqrt{(1+r^2)^3} \int_{\bar \varepsilon(r)}^\infty \frac{\ext\!y}{\big[y(1+Ay+By^2)\big]^3}\, ,
$$
where we made the variable substitution~$y=x-\frac12 r^2$, and
used that
$$
\bar \sigma= \frac{y}{\sqrt{1+r^2}}
\, \big(1+Ay+By^2\big)\ \mbox{ with }\
A:=-\tfrac14\tfrac{3r^2+2}{(1+r^2)^2}\  \mbox{ and } \  
B:=\tfrac{1}{6}\tfrac{r^2 (5r^2+6)}{(1+r^2)^4}\, .$$
The~$y$ integration can be performed explicitly, and then expanded in a series about~$y=0$, the leading terms are quoted below
$$
\int \frac{\ext\!y}{\big[y(1+Ay+By^2)\big]^3}=- \frac1{2y^2}
+\frac{3A}{y}+6\Big(A^2-\frac{B}2\Big)\, \log y + {\mathcal O}(y^0)\, .
$$
Since we are only interested in the divergences and log behavior, we only care about the dependence of the~$y$ integral on its lower terminal. Thus by substituting~$y=\bar\varepsilon(r)$ and expanding in powers of~$\varepsilon$ we have
$$
\int_{\bar \varepsilon(r)}^\infty \frac{\ext\!y}{\big[y(1+Ay+By^2)\big]^3}=
\frac{1}{2\varepsilon^2 (1+r^2)}+
\frac{3r^2+2}{2\varepsilon (1+r^2)^{5/2}}
- \log\varepsilon\, \frac{7 r^4+12 r^2 + 12}{8(1+r^2)^4}
+{\mathcal O}(\varepsilon^0)\, .
$$
The radial integrations are now easily performed and yield
\begin{equation*}
\begin{split}
\Vol_\varepsilon\ =\ \frac{\pi\, \big(\sqrt{(1+R^2)^3}-1\big)}{\varepsilon^2}
&\ +\ \frac{\pi\, \big(3 R^2-\log(1+R^2)\big)}{2\varepsilon}\\[1mm] & \ +\ 
\frac{5\pi}3\, \log\varepsilon\, \Big(1-\frac{1+\tfrac{12}{5}R^2+
\tfrac{21}{20}R^4}{\sqrt{(1+R^2)^3}}\Big)+{\mathcal O}(\varepsilon^0)\, .
\end{split}
\end{equation*}
This result matches exactly the divergences
in Equation~\nn{expdiv}
that were generated by Theorem~\ref{divergences}
 (or equivalently Equation~\ref{div3}),
and the anomaly in Equation~\nn
{expA} that was computed from the result of Theorem~\nn{singanomaly}.

%

\newcommand{\msn}[2]{\href{http://www.ams.org/mathscinet-getitem?mr=#1}{#2}}
\newcommand{\hepth}[1]{\href{http://arxiv.org/abs/hep-th/#1}{arXiv:hep-th/#1}}
\newcommand{\maths}[1]{\href{http://arxiv.org/abs/math/#1}{arXiv:math/#1}}
\newcommand{\mathph}[1]{\href{http://lanl.arxiv.org/abs/math-ph/#1}{arXiv:math-ph/#1}}
\newcommand{\arxiv}[1]{\href{http://lanl.arxiv.org/abs/#1}{arXiv:#1}}

\end{document}